\documentclass{amsart}
\usepackage[utf8]{inputenc}
\usepackage{verbatim}
\usepackage{amsfonts}
\usepackage{mathtools}
\usepackage[english]{babel}
\usepackage{xcolor}
\usepackage[T1]{fontenc}
\usepackage{amsmath}
\usepackage{amsthm}
\usepackage{float}
\usepackage{wrapfig}
\usepackage{tabu}

\usepackage{MnSymbol}
\usepackage{tcolorbox}
\usepackage{tikz}
\usepackage{young}
\usepackage[all]{xy}
\usepackage{pifont}

\usepackage{stackrel}
\newcommand{\acts}{\curvearrowright}
\usepackage{mathtools}
\usepackage{mathrsfs}

\usetikzlibrary{arrows}
\usepackage[a4paper]{geometry}

\usepackage{amsthm}
\usepackage{amsfonts}
\usepackage{graphicx}
\usepackage[colorinlistoftodos]{todonotes}
\usepackage[colorlinks=true, allcolors=blue]{hyperref}
\newcounter{theoremintro}
\newtheorem{thmintro}[theoremintro]{Theorem}

\newtheorem{corintro}[theoremintro]{Corollary}

\newtheorem{defintro}[theoremintro]{Definition}

\newtheorem{theorem}{Theorem}[section]
\newtheorem{corollary}[theorem]{Corollary}
\newtheorem{lemma}[theorem]{Lemma}
\newtheorem{rmk}[theorem]{Remark}
\newtheorem{example}[theorem]{Example}

\newtheorem{definition}[theorem]{Definition}

\newtheorem{question}[theorem]{Question}
\newcommand{\Set}[1]{\{ #1 \}}

\usepackage{dirtytalk}

\newcommand{\G}{\mathcal{G}}

\newcommand{\B}{\mathcal{B}}

\title{Studying Stein's Groups as Topological Full Groups}
\author{Owen Tanner}
\date{\today}
\newcommand{\Gn}[1]{\mathcal{G}^{(#1)}}

\begin{document}
\newpage

\maketitle
\section*{Abstract}
We include a class of generalisations of Thompson's group $V$ introduced by Melanie Stein into the growing framework of topological full groups. Like $V$, Stein's groups can be described as certain piecewise linear maps with prescribed slopes and nondifferentiable points. Stein's groups include the class of groups where the group of slopes is generated by multiple integers (which we call Stein's integral groups) and when the group of slopes is generated by a single irrational number, such as Cleary's group, (which we call irrational slope Thompson's groups). We give a unifying dynamical proof that whenever the group of slopes is finitely generated, the simple derived subgroup of the associated Stein's groups is finitely generated. We then study the homology of Stein's groups, building on the work of others. This line of inquiry allows us to compute the abelianisations and rational homology concretely for many examples.  

\section{Introduction}

Thompson's group $V$ was the first known example of an infinite finitely presented simple group \cite{thompson}. Since the introduction and dissemination of Thompson's work into the mathematical community, $V$ has impacted mathematics in countless often disconnected and surprising ways, leading Brin to call $V$ the Chameleon group for its breadth of impact \cite{brin1996chameleon}. Combinatorial group theorists have been interested in Thompson's group $V$ because of the connection to interesting open questions such as the Boone-Higman conjecture \cite{belk2023progress}. Computer scientists have been interested in Thompson's group from the perspective of cryptography, as a platform group \cite{shpilrain2005thompson}. More recently, it has attracted the attention of those working in C$^*$-algebras and ample groupoids because of the deep connection to the Cuntz algebra \cite{nekrashevych2004cuntz}, via topological full groups \cite{matui2012homology}, \cite{matui2014topological}. 

A family of groups were introduced by Melanie Stein in 1992  \cite{stein1992groups}. This class vastly generalises Thompson's group $V$ but still satisfies many of the properties that $V$ satisfies. For example, they provide a source of simple derived subgroups which often are finitely generated. Let us briefly recall her construction. 

\begin{defintro}[Stein's groups]
    Let $\Lambda$ be a subgroup of $(\mathbb{R}_+, \cdot)$. Let $\Gamma$ be a subgroup of the group ring $(\mathbb{Z}\cdot\Lambda ,+)$. Let $\ell \in \Gamma$. Then, Stein's group 
    $V(\Gamma,\Lambda,\ell)$ associated to the triple $(\Gamma,\Lambda,\ell)$, is the group of piecewise linear bijections of $[0,\ell]$, with finitely many slopes, all in $\Lambda$ and finitely many nondifferentiable points, all in $\Gamma$. 
\end{defintro}
This family also encompasses many interesting examples that have been studied in great detail such as:
\begin{itemize}
\item Thompson's group $V$, which corresponds to choosing $\Lambda=\langle 2 \rangle , \Gamma=\mathbb{Z}[1/2]$ and $ \ell=1$. 
 \item The Higman-Thompson groups $V_{k,r}$ \cite{higman1974finitely} where $k,r \in \mathbb{N}$, which correspond to choosing $\Lambda=\langle k \rangle, \, \Gamma=\mathbb{Z}[1/k], $ and $ \ell=r$.
 \item Stein's integral groups, which corresponds to choosing a finite collection of integers $\{n_1,...,n_k\}$ and a length $r \in \mathbb{N}$, and taking $\Lambda=\langle n_1,...,n_k \rangle, \, \Gamma=\mathbb{Z} \cdot \Lambda,$ and $ \ell=r$. 
    \item $V_\tau$, Cleary's group \cite{cleary2000regular}, \cite{Clearymore}, \cite{burillo2021irrationalslope}, also known as the irrational slope Thompson group \cite{irrationalslope22}, which corresponds to choosing $\Lambda=\langle \tau \rangle, \Gamma=\mathbb{Z} \cdot \Lambda,$ and $ \ell=1$ for $\tau=\frac{\sqrt{5}-1}{2}$, and other related irrational slope Thompson groups (corresponding to different choices of $\tau$).
   
\end{itemize} 

Stein's groups are the primary object of study in this paper but we take a different perspective from other authors who have studied Stein's groups previously \cite{stein1992groups}, \cite{irrationalslope22}, \cite{martinez2016cohomological}. Namely, we take the dynamical perspective; that Stein's groups are topological full groups, in the sense of Matui \cite{matui2012homology}. The key idea of Matui's framework is that simple groups with finiteness properties are often built from kinds of partial symmetries that are best captured by bisections in ample groupoids. This is paired with a philosophy that the ample groupoids themselves are often more accessible to study than simple groups, be that from the perspective of homology \cite{li2022}, finite generation \cite{nekrashevych2019simple} or amenability \cite{juschenkomonod}, \cite{juschenko2016extensions}. This relative accessibility is partly because minimal ample groupoids have been studied extensively by mathematicians working in C$^*$-algebras for many decades, see the survey of Sims \cite{sims2017etale} or books of Paterson \cite{paterson2012groupoids}, Renault \cite{renault2006groupoid} for an overview on this research. This philosophy has since been strengthened by the remarkable results of Nekrashevcych \cite{nekrashevych2019simple}, and Li \cite{li2022}, which transfer algebraic and homological information from the ample groupoids to the topological full groups. 

This example class is not an exception to the philosophy that ample groupoids are often more accessible to study. These ample groupoids can be described as full corners in the universal groupoid of the inverse hull of $\Gamma_+ \ltimes \Lambda$, which, as shown in work of Li \cite{li2012nuclearity}, we could equally describe as the partial transformation groupoid of $\Gamma \ltimes \Lambda$ on the Cantor space $X$, which is the description which we give throughout Section 4. See Lemma \ref{v as a topological full group} for the precise construction. This perspective is convenient because the partial action $\beta: \Gamma \ltimes \Lambda \acts X$ is minimal and topologically free. This gives us the simplicity of the derived subgroups of these groups by using standard results in the literature of topological full groups \cite{matui2014topological}.

The partial action $\beta$ is also expansive, in the sense of classical dynamics. We show that this notion of expansivity for partial actions is closely related to Nekrashevych's notion of expansivity for groupoids \cite{nekrashevych2019simple}, even equivalent in the case of compactly generated partial transformation groupoids. We also study the property of compact generation for groupoids, showing that it is preserved under taking full compact open corners. This line of inquiry allows us to show that the derived subgroups are finitely generated under very general circumstances, which is the main result of our paper:
\begin{thmintro}[Theorem \ref{fg when fg by alg}]

   Let $\Lambda$ be a subgroup of $(\mathbb{R}_+, \cdot)$ and $\Gamma$ be a $\mathbb{Z} \cdot \Lambda$-submodule and $\ell \in \Gamma$. Then, $D(V(\Gamma,\Lambda,\ell))$ is simple. Moreover, the following are equivalent:
   \begin{enumerate}
   
       \item $\Gamma \ltimes \Lambda$ is finitely generated.
       \item $\Gamma \ltimes \Lambda \ltimes [0_+,\ell_-]$ is compactly generated for all $\ell \in \Gamma$. 
       \item $D(V(\Gamma,\Lambda,\ell))$ is  finitely generated for all $\ell$. 
        \item $D(V(\Gamma,\Lambda,\ell))$ is  2-generated for all $\ell$. 
  \end{enumerate}
  \label{thmintro1}
\end{thmintro}
This generalises the known case due to Stein, who showed that $D(V(\Gamma,\Lambda,\ell))$ is simple and finitely generated in the case when $\Lambda$ is generated by finitely many integers, $\Gamma=\mathbb{Z} \cdot \Lambda$, and $\ell \in \mathbb{N}$ \cite{stein1992groups} who herself generalised work of Higman \cite{higman1974finitely}. It also generalises the result of Burillo-Nucinkis-Reeves who showed $D(V(\Gamma,\Lambda,\ell))$ is simple and finitely generated in the case $\Lambda=\langle \frac{1+ \sqrt{5}}{2} \rangle $, $\Gamma=\mathbb{Z} \cdot \Lambda$ and $\ell=1$ \cite{irrationalslope22}. In particular, a consequence is that for all choices of irrational number $\lambda$, the irrational slope Thompson group $V(\mathbb{Z}[\lambda,\lambda^{-1}], \langle \lambda \rangle, 1)$ has a simple and finitely generated derived subgroup.

In Section 6. we study the homology of Stein's groups. Since the conception of topological full groups, there has been a deep connection to groupoid homology, most notably in Matui's AH conjecture \cite{matui2012homology}, \cite{matui2014topological}. This conjecture was recently confirmed by Li \cite[Corollary E]{li2022} under very minor regularity conditions which are satisfied by our groupoid model for Stein's groups. Inspired by Szymik-Wahl's work on the homology of Higman-Thompson groups \cite{szymik2019homology}, the framework of Li \cite{li2022} found other connections between the group homology of a topological full group to the underlying groupoid's homology. In the case of Stein's groups, the groupoid homology is comparatively computable, making the homology of Stein's groups accessible for the first time.

Using this homology computation in combination with \cite[Corollory E]{li2022}, we show that the abelianisation of $V(\Gamma,\Lambda,\ell)$ is finite rank for the case of $\Gamma,\Lambda$ generated by finitely many algebraic numbers (Lemma \ref{fg v type}). Combining this with Theorem \ref{fg when fg by alg} we prove a finite generation theorem for the $V$-type groups:
\begin{thmintro}
   Let $\Lambda$ be a subgroup of $(\mathbb{R}_+, \cdot)$ generated by finitely many algebraic numbers. Let $\Gamma$ be a submodule of the group ring $(\mathbb{Z} \cdot \Lambda,+)$. Let $\ell \in \Gamma$. Then,  $V(\Gamma,\Lambda,\ell)$ is finitely generated. \label{thmintro2}
\end{thmintro}

These homology computations are especially interesting from the perspective of homological stability since we show that the natural inclusion of Stein's groups acting upon compact intervals into the noncompact Stein groups (say coming from the inclusion of $[0,\ell] $ in $\mathbb{R}$) induces an isomorphism on the level of homology. We consider this to be a generalisation of \cite[Theorem 3.6]{szymik2019homology}. 

\begin{corintro}(Corollary \ref{cor hom best})
Let $\Lambda$ be a subgroup of $(\mathbb{R},\cdot)$, $\Gamma$ be a submodule of $(\mathbb{Z} \cdot \Lambda,+)$. Let $U$ be any closed subset of $\mathbb{R}$ with nonempty interior. Let $V(\Gamma,\Lambda, U)$ be the group of piecewise linear bijections of $\mathbb{R}$ with finitely many slopes (all in $\Lambda$) and finitely many nondifferentiable points (all in $\Gamma$) that are the identity on $U^c$. Then for all $* \in \mathbb{N}$:
$$H_*(V(\Gamma,\Lambda,U)) \cong H_*(V(\Gamma,\Lambda,\mathbb{R}))$$
\label{corhomintro}
\end{corintro}For more specialised homology results, we turn our focus onto two subclasses of Stein's groups. The first subclass is the case where the group of slopes is cyclic, we refer to these groups as irrational slope Thompson groups for consistency with the literature. 

The homology of irrational slope Thompson groups is related to the homology of certain ample groupoids which was computed by Li in \cite{xinlambda}. This groupoid homology is highly sensitive to the minimal polynomial of the underlying generator $\lambda$, which leads the irrational slope Thompson group construction to exhibit diverse homological behaviour in the algebraic case. In Corollary \ref{abelianisation single lambda}, we compute the abelianisation $V(\mathbb{Z}[\lambda,\lambda^{-1}], \langle \lambda \rangle,\ell))_{ab}$ for low degree algebraic numbers $\lambda$ explicitly, generalising and unifying the known results due to Higman \cite{higman1974finitely} (who computed the abelianisation for $\lambda \in \mathbb{N}$) and Burillo-Nucinkis-Reeves \cite{irrationalslope22} (who computed the abelianisation to be $\mathbb{Z}_2$ for $\lambda=\frac{\sqrt{5}-1}{2}$).

We also compute the rational homology explicitly for certain irrational slope Thompson groups in Theorem \ref{rational hom comp}. Notably, the Higman-Thompson groups were shown to all be rationally acyclic \cite[Corollary C]{szymik2019homology} by Szymik-Wahl. The irrational slope Thompson groups vary dramatically:
\begin{itemize}
    \item Cleary's group $V_\tau$ is rationally acyclic.
    \item There are irrational slope Thompson groups which are not rationally acyclic, or even virtually simple. For example, if we take $\lambda=\frac{3+\sqrt{5}}{2}$, then $V(\mathbb{Z}[\lambda,\lambda^{-1}],\langle \lambda \rangle, 1)_{ab}=\mathbb{Z}$ by Corollary \ref{abelianisation single lambda}.  
\end{itemize}

 Groupoid homology itself is an invariant for topological full groups via Matui's isomorphism theorem. This allows us to distinguish different examples of Stein's groups. For example, see Corollary \ref{cor classification}, which gives an invariant for Stein's groups in terms of $\Gamma,\Lambda$ and $\ell$. This invariant is fine enough to give one direction in the classification of Higman-Thompson groups \cite{pardo2011isomorphism} and in combination with our explicit  
homology computations for irrational slope Thompson groups enable us to prove an analogue to Higman's result \cite{higman1974finitely} giving one direction on the classification of Higman-Thompson groups \cite{pardo2011isomorphism}, but generalised to algebraic numbers of degree less than or equal to $2$, and more arbitrary lengths of the compact intervals $\ell$:
\begin{corintro}[Corollary \ref{classification for low degree}]
      Let $\lambda,\mu<1$ be algebraic numbers with degree $\leq 2$ and let $\ell_1 \in \mathbb{Z}[\lambda,\lambda^{-1}], \ell_2 \in  \mathbb{Z}[\mu,\mu^{-1}]$. Suppose that
     $ V(\mathbb{Z}[\lambda,\lambda^{-1}], \langle \lambda \rangle,\ell_1) \cong V(\mathbb{Z}[\mu,\mu^{-1}],\langle \mu \rangle,\ell_2)$.
     
    Then, $\lambda=\mu$, and $\ell_1-\ell_2 \in (1-\lambda)\mathbb{Z}[\lambda,\lambda^{-1}]$.  \label{corintroclass}
\end{corintro}

The second subclass we study in detail is the class of examples for which $\Lambda$ is generated by several integers, which we call Stein's integral groups. To compute the groupoid homology for this case, we rely on the framework of $k$-graphs, identifying our groupoid with the groupoid of a single vertex $k$-graph. The observation that we can rephrase the groupoid model of these particular groups in the language of k-graphs is also noted in upcoming work by Conchita Martınez-Pérez, Brita Nucinkis and Alina Vdovina, we include it here for completeness in the literature and for our homology computations. Most notably, this allows us to use the groupoid homology computations of Farsi-Kumjian-Pask-Sims \cite{kgraphcomp} to obtain acyclicity results on the level of Stein's integral groups. 
\begin{corintro}[Corollary \ref{acyclicty for stein}]
 Let $n_1,..,n_k$ be a finite collection of integers. Let $\ell \in \mathbb{Z}[\frac{1}{n_1 n_2 ... n_k}]$. Then $V(\mathbb{Z}[\frac{1}{n_1 n_2 ... n_k}], \langle n_1,n_2,...,n_k \rangle ,\ell)$ is rationally acyclic. 

 $V(\mathbb{Z}[\frac{1}{n_1 n_2 ... n_k}], \langle n_1,n_2,...,n_k \rangle ,\ell)$ is integrally acyclic if and only if $gcd(n_1-1, ..., n_k-1)=1$. \label{cor intro acyclic}
\end{corintro}
This generalises the result of Szymik-Wahl \cite[Corollary C]{szymik2019homology} that the Higman-Thompson groups are rationally acyclic and that Thompson's group $V$ is acyclic. The key tool we use is the two acyclicity results in the paper of Li \cite[Corollary C, Corollary D]{li2022}

\textbf{Acknowledgements:} This project forms part of the PhD thesis of the author. This PhD is supervised by Xin Li, who I would like to thank for the supervision. I want to thank Alistair Miller, particularly for his useful insights concerning the compact generation of ample groupoids. I want to thank Jim Belk and Brita Nucinkis for conversations about Thompson-like groups and topological full groups. I would also like to thank Anna Duwenig for useful conversations about Zappa-Szep products. The author has received funding from the European Research Council (ERC) under the European Union’s Horizon 2020 research and innovation programme (grant agreement No.
817597). 
\section{Preliminaries}
\textbf{Notation:} We reserve $\sqcup$ for disjoint unions of sets. We take the convention that $0 \nin \mathbb{N}$. 
\subsection{Groupoids and Partial Actions}
A \textit{groupoid} is a small category of isomorphisms. This is a set $\G$ with partially defined multiplication $\gamma_1 \gamma_2$ and everywhere defined involutive operation $\gamma \mapsto \gamma ^{-1}$, satisfying:
\begin{enumerate}
    \item Associativity: If $\gamma_1\gamma_2$ and $(\gamma_1\gamma_2)\gamma_3$ are defined, then $\gamma_2\gamma_3$ is defined and $(\gamma_1\gamma_2)\gamma_3=\gamma_1(\gamma_2\gamma_3)$
    \item Existence of $r,s$: The range and source maps $r(\gamma)=\gamma\gamma^{-1}$, $s(\gamma)=\gamma^{-1}\gamma$ are always well defined. If $\gamma_1\gamma_2$ are well defined, then $\gamma_1=\gamma_1\gamma_2 \gamma_2^{-1}$ and $\gamma_2=\gamma_1^{-1} \gamma_1 \gamma_2$. 
\end{enumerate}
A \textit{topological groupoid} is a groupoid endowed with a topology such that the multiplication and inverse operations are continuous.

Elements of the form $\gamma\gamma^{-1}$ are called units and the space of all units is denoted $\G^{(0)}$. The domain of the multiplication map is called the space of composable pairs and is denoted by $\Gn{2}$. It is known that $\Gn{2}=\{(\gamma_1,\gamma_2) \in \G^2 \; s(\gamma_2)=r(\gamma_1) \}$. Given two subsets $B_1,B_2 \subset \G$, we define their product to be $B_1 B_2:=\Set{\gamma_1\gamma_2 \; : \; \gamma_i \in B_i, \; (\gamma_1,\gamma_2) \in \Gn{2}} $

Units $u_1,u_2 \in \Gn{0}$ belong to the same $\G$\textit{-orbit} if there exists $\gamma \in \G$ such that $s(\gamma)=u_1, \; r(\gamma)=u_2$, and the orbit of $u$ is denoted $\G(u)$. If $\G(u)$ is dense in $\Gn{0}$ for all $u \in \Gn{0}$ we call the groupoid \textit{minimal}.

The \textit{isotropy group}  (denoted $\G_u$) of a unit $u \in \Gn{0}$ is the group $\Set{\gamma \in \G \; : \; s(\gamma)=r(\gamma)=u}$ if we have that every isotropy group is trivial, i.e. $\G_u=\Set{u}$ one says that the groupoid is \textit{principal}. A weaker condition is \textit{topologically principal} which means that such units are dense in the unit space  ($\overline{ \Set{u \in \Gn{0} \; : \; \G_u= \Set{u}}}=\Gn{0}$). We now give one of our key examples for this paper- a \textit{transformation groupoid}. 
\begin{example}
Let $\Gamma \acts X$ be a group acting by homeomorphisms on a topological space $X$. Then we define the \textit{transformation groupoid} $\Gamma \ltimes X$ to be the set of pairs $(g,x) \in \Gamma \times X$ here composable pairs are of the form $(g,h(x))(h,x)$ and composition is given by $(g,h(x))(h,x)=(gh,x)$. Here $s(g,x)=(1,x), \; r(g,x)=(1,g(x))$ and the unit space is canonically identified with $X$. A basis of the topology is given by $(g.U)$ where $g \in \Gamma$. and $U$ is open in $X$. This makes the topology of the transformation groupoid Hausdorff. Note that we have many properties of actions translate into properties of the transformation groupoid:
\begin{itemize}
\item If the action is\textit{ free} (i.e. for all $g \in \Gamma, x \in X$ $gx=x \implies g=1$ ) if and only if the groupoid is principal. 
\item If the action is \textit{topologically free} (i.e. for all $g \in \Gamma \setminus \Set{ 1}$, the set of fixed points ${x \in X \; : \; gx=x }$ has empty interior) if and only if the groupoid is topologically principal. 
    \item If the action is minimal (i.e. the orbit $Gx$ is dense in $X$ for all $x \in X$), if and only if the groupoid is minimal. 
\end{itemize}

\end{example}
Another construction we study in this paper is a subtly different groupoid; the restriction of a transformation groupoid to a partial action. 
\begin{example}[Partial Action]
Let $\Gamma \ltimes Y$ be a transformation groupoid. Let $X \subset Y$ be an open subset such that every orbit $\Gamma x, \; x \in Y$ has nontrivial intersection with $X$.  Then let

$$\Gamma \ltimes Y |_X^X:=\{ (\gamma,x) \in \Gamma \ltimes Y \; : \; x, \gamma(x) \in X \} $$
This forms a subgroupoid of the transformation groupoid. We call this a partial action $\alpha: \Gamma \acts X$ on $X$ and denote the associate groupoid $\Gamma \ltimes_\alpha X$. 

   Let $\alpha: \Gamma \acts X$ be a minimal, topologically free partial action.  As a groupoid $\Gamma \ltimes_\alpha X$ is an amenable, topologically free, Hausdorff and minimal groupoid. $\alpha$ is free if and only if $\Gamma \ltimes_\alpha X$ is principal.
\end{example}

A crucial tool for understanding topological groupoids comes from studying their bisections. A class of groupoids whose bisections generate the topology, namely \'etale groupoids, are of particular interest to people working in discrete dynamics. These are the generalisation of transformation groupoids by discrete groups.  

A $\G$\textit{-bisection} is an open subset $B \subset \G$ such that $ s: B \rightarrow s(B) \; r:B \rightarrow r(B)$ are homeomorphisms and $s(B), r(B)$ are open. We denote the space of open bisections by $\mathcal{B}$. A topological groupoid $\G$ is said to be \textit{ \'etale } if $\mathcal{B}$ forms a basis of the topology on $\G$.

 Equivalently, a groupoid is \'etale if the maps $r,s$ are local homeomorphisms. In the \'etale case, the set $\B$ is an inverse semigroup with respect to set multiplication and pointwise inversion. Informally, one thinks of \'etaleness as a discreteness property, because for (partial) transformation groupoids they are \'etale if and only if the acting group is discrete.  
 $\G$ is said to be \textit{ample} if, in addition to being \'etale, the unit space is totally disconnected, in particular, 
an ample groupoid $\G$ is said to be a Cantor groupoid if its unit space $\Gn{0}$ is a Cantor space. In this case, the set of \textit{compact open bisections }, denoted $\mathcal{B}^k$

An \'etale groupoid $\G$ is said to be a \textit{effective} if for all $g \in \G \setminus \Gn{0}$ and all $B \in \B$ containing $g$, there exists some $g \in B$ such that $s(g) \neq r(g)$. This is weaker than being topologically principal in general, but the notion agrees for Hausdorff groupoids; in particular for transformation groupoids. 
\begin{example}
    Let $\alpha: \Gamma \acts X$ be a partial action of a discrete group on a locally compact Hausdorff space. Then $\Gamma \ltimes X$ is \'etale. Moreover, a basis of the open bisections is given by: 
    $$(g,U) \; \; g \in \Gamma, \; \; U,g(U) \subset X \text{ open}$$
    $\Gamma \ltimes X$ is ample if and only if $X$ is totally disconnected. In this case, a basis for the open compact bisections is given by:
    $$(g,U) \; \; g \in \Gamma, \; \;  U,g(U) \subset X \text{ compact, open}$$
 
\end{example}
Note that if $U,V$ are compact open bisections, such that $s(U) \cap s(V)=r(U)\cap r(V) = \emptyset$ then $U \sqcup V$ is a (compact) open bisection. A necessary condition for the associated C$^*$-algebras to be purely infinite is phrased in terms of compact open bisections \cite{matui2016etale}. 
\begin{definition}[Purely Infinite]
    Let $\G$ be an effective ample groupoid. We say that $\G$ is purely infinite if for all $A \subset \G^{(0)}$ compact open, there exists $B,B' \in \mathcal{B}^k$ such that $S(B)=s(B')=A$, $r(B)\cap r(B')=\emptyset$, $r(B) \sqcup r(B') \subset A$. 
    \label{purely infinite}
\end{definition}

\subsection{Topological Full Groups}
Using the concept of bisections, we are ready to define the topological full group of a Cantor groupoid $\G$:
\begin{definition}[Topological Full Group (as in \cite{nekrashevych2019simple}, Definition 2.3)]
Let $\G$ be a Cantor groupoid. The topological full group, denoted $\mathsf{F}(\G)$ is the group of bisections $$ \mathsf{F}(\G)=\{\gamma \in \mathcal{B}^k \; : \;  s(\gamma)=r(\gamma)=\Gn{0} \}$$ with respect to pointwise multiplication.  \label{groupoid tfg}
\end{definition}
In other words, this is the unital subgroup $U(\mathcal{B}^k)$ of the inverse monoid of compact open bisections $\mathcal{B}^k$. Elements $B$ of topological full groups can also be thought of as homeomorphisms $f_B$ of the unit space:
$$f_B=(r_{\restriction_B})\circ (s_{\restriction_B})^{-1}: (\G)^{(0)} \rightarrow (\G)^{(0)}$$
If $\G$ is effective, the map $f: \mathsf{F}(\G) \rightarrow Homeo(\G^{(0)}) \quad B \mapsto f_B$ is an injection.

This movement between perspectives is routinely used throughout this text, and so we often assume that our groupoids are effective. Let us discuss what happens outside the effective case for completeness of the literature. 

Suppose your groupoid is not effective. Then there is a canonical normal subgroup inside $\mathsf{F}(\G)$, given by the kernel of $f$. 
$$ \mathsf{K}^f(\G)=\{B \in \mathsf{F}(\G) \; : \; \forall \gamma \in B \; s(\gamma)=r(\gamma) \} $$
This normal subgroup is related to a well known construction in the theory of \'etale groupoids, the notion of taking a quotient called the groupoid of germs. If $\G_{germ}$ is the groupoid of germs of a Cantor groupoid $\G$ then $\mathsf{F}(\G)/\mathsf{K}^f(\G)=\mathsf{F}(\G_{germ})$. The groupoids we consider in this texts are partial transformation groupoids by essentially free actions, and so they are effective. This allows us to describe their topological full groups in a variety of ways. 
\begin{example}

   Let $\Gamma \ltimes_\alpha X$ be the partial transformation groupoid of a discrete group acting on the Cantor space. Suppose that $\alpha$ is essentially free. Then the following groups are isomorphic:
   \begin{itemize}
       \item $\mathsf{F}(\Gamma \ltimes_\alpha X)$
       \item  The group of homeomorphisms $ \gamma \in Homeo(X)$ such that $ \forall x \in X, \exists g \in \Gamma, U \subset X$ compact, open neighbourhood of $x \text{ such that } \gamma|_U=f|_U$. 
       \item The group of homeomorphisms $\gamma \in Homeo(X) $ such that there exists a finite partition $X=\bigsqcup_{i=1}^n X_i$ into compact open subsets, and  $ g_1,...g_n \in \Gamma \text{ such that } \gamma|_{X_i}=\alpha(g_i)|_{X_i}$.
   \end{itemize}\label{description of a tfg of a partial action}    
\end{example}
Let us discuss the subgroup structure of topological full groups. Let $\G$ be an effective, Cantor groupoid with infinite orbits. Let us define an analogue of the infinite alternating group in $\mathsf{F}(\G)$, which we call the alternating group of $\G$. For open compact bisection $B_1,B_2$ with $s(B_1),r(B_1)=s(B_2),r(B_2)$ pairwise disjoint, let 
$$\gamma_{B_1,B_2}=B_1 \sqcup B_2 \sqcup (B_1 B_2)^{-1} \sqcup (\G^{(0)} \setminus s(B_1) \sqcup s(B_2) \sqcup r(B_2)) \in \mathsf{F}(\G)$$

We define $$\mathsf{A}(\G)=\langle \gamma_{B_1,B_2} \; : \;  B_1,B_2 \in \mathcal{B}^k, \; s(B_1),r(B_1)=s(B_2),r(B_2) \text{ are pairwise disjoint } \rangle $$ 

   Another important subgroup of a topological full group is it's derived subgroup. Let $\G$ be an effective Cantor groupoid. Let $\mathsf{D}(\G)$ denote the derived subgroup of $\mathsf{F}(\G)$, that is $\mathsf{D}(\G)=\langle [\gamma_1,\gamma_2] \; : \; \gamma_1,\gamma_2 \in \mathsf{F}(\G) \rangle $

Note in particular then $\mathsf{A}(\G)$ is a subgroup of $\mathsf{D}(\G)$. Indeed if we take for $B \in \mathcal{B}^k$ with $s(B) \cap r(B) = \emptyset$ 
$$ \gamma_{B}:=B \sqcup B^{-1} \sqcup (\G^{0} \setminus s(B) \cup r(B) ) \in \mathsf{F}(\G)$$
Then $ \gamma_{B_1,B_2}=[\gamma_{B_1},\gamma_{B_2}].$ It is currently open whether there exists an effective, \'etale, Cantor groupoid $\G$ such that $\mathsf{A}(\G) \neq \mathsf{D}(\G)$. 

The above group elements $\gamma_B$ again generate a certain subgroup. Let us define an analogue of the infinite symmetric group in $\mathsf{F}(\G)$, which we call the symmetric group of $\G$. Then $$\mathsf{S}(\G)=\langle \gamma_B \; : \; B \in \mathcal{B}^k \; s(B) \cap r(B)=\emptyset \rangle $$

Let us remark that each of the groups we have defined thus far are normal in $\mathsf{F}(\G)$. The normality of $\mathsf{D}(\G)$ is generic for such groups, but for $\mathsf{A}(\G),\mathsf{S}(\G)$ it follows from the observation that for all $\gamma \in \mathsf{F}(\G)$, and $\gamma_B \in \mathsf{S}(\G)$, $\gamma \gamma_B \gamma^{-1}=\gamma_{\gamma B \gamma^{-1}}$. 
Simplicity of $\mathsf{A}(\G)$ is known to be equivalent to minimality of $\G$.

\begin{theorem}[Matui \cite{matui2014topological}, Theorem 4.7, \cite{nekrashevych2019simple}, Theorem 1.1]
Let $\G$ be purely infinite, effective Cantor groupoid. Then, the following are equivalent:  \label{d simple}
\begin{itemize}
    \item $\G$ is minimal. 
    \item $\mathsf{D}(\G)$ is simple. 
\end{itemize}
Moreover, in the case where $\G$ is minimal, $\mathsf{D}(\G) \cong \mathsf{A}(\G)$. 
\end{theorem}
A crucial deep result in the theory of topological full groups is that we do not lose any information when passing from $\G$ to $\mathsf{F}(\G)$ in the following sense. 
\begin{theorem}[Matui-Rubin Isomorphism Theorem \cite{matui2014topological}, Theorem 3.10]
\label{matui isomorphism theorem}
Let $\G_1,\G_2$ be essentially principal, \'etale, minimal Cantor groupoids. Then the following are equivalent:
\begin{itemize}
    \item $\G_1 \cong \G_2$ as \'etale groupoids. 
    \item $\mathsf{F}(\G_1) \cong \mathsf{F}(\G_2)$ as discrete groups.
     \item $\mathsf{D}(\G_1) \cong \mathsf{D}(\G_2)$ as discrete groups.
\end{itemize}
\end{theorem}
Two partial actions $\alpha: G \acts X$ $\beta: H \acts Y$ give rise to the same partial transformation groupoids if and only if they are continuous orbit equivalent, as noted by Li in \cite{li2018continuous}.

On some occasions, we generalize the definition of a topological full group to include groupoids whose unit space is not a Cantor space. Such work has been pioneered by Nyland-Ortega \cite{Nyland_2019}, who focused on general graph groupoids \cite{Nyland_2019}, \cite{nyland2021matui} and Katsura-Exel groupoids \cite{nyland2021katsura}.
Amongst their general results, they generalised Theorem \ref{d simple} and Theorem \ref{matui isomorphism theorem} to the noncompact setting. The definition of a topological full group in this setting due to Nyland and Ortega and is as follows:
\begin{definition}
    \label{ortega nyland}

    Let $\G$ be an effective ample groupoid. For each compact open bisection $B \in B^k$ in $\G$ such that $s(B)=r(B)$, let us associated a homeomorphism of $\G^{(0)}$:
    $$\pi_B: \G^{(0)} \rightarrow \G^{(0)} \quad x \mapsto \begin{cases}
        Bx & x \in s(B) \ \\
        x & x \nin s(B)
    \end{cases}$$
    We define $F(\G)$ to be the subgroup of $Homeo(\G^{(0)})$ consisting of all such homeomorphisms. 
    $$F(\G)=\{ \pi_B \; : B \in B_\G^k, s(B)=r(B) \} $$
    With respect to composition. 
\end{definition}

The groupoids we study are purely infinite and minimal. We end by recalling the main result of \cite{gardella2023generalisations}, which characterizes pure infiniteness and minimality of a groupoid through properties of its topological full group. Let us first define these properties. 
\begin{definition}[Vigor]
(\cite[Definition 1.1]{vigorous}). Let $\Gamma \leq \mathrm{Homeo}(X)$ be a subgroup of homeomorphisms of the Cantor space $X$. We say that $\Gamma$ is \emph{vigorous} if whenever 
$U, Y_1,Y_2 \subseteq X$ are compact and open with $U\neq X$ and $Y_2\neq \emptyset$, and satisfy $Y_1,Y_2 \subseteq U$, then there exists $g \in \Gamma$ such that $g$ is the identity on $X \setminus U$, and $g(Y_1) \subseteq Y_2$. 
\end{definition}
\begin{definition}[Compressible Action (as in  \cite{dudko2014finite})]
An action of a discrete group $\Gamma$ on a locally compact Hausdorff space $X$ is said to be 
\emph{compressible}, if there exists a subbase $\mathcal{U}$ for the topology on $X$ such that: \begin{enumerate}
    \item for all $ g \in \Gamma$, there exists $U \in \mathcal{U}$ such that $\mathrm{supp}(g) \subseteq U$;
    \item for all $ U_1, U_2 \in \mathcal{U}$, there exists $g \in \Gamma$ such that $g(U_1) \subseteq U_2$;
    \item for all $ U_1,U_2,U_3 \in \mathcal{U}$ with $\overline{U}_1 \cap \overline{U}_2 = \emptyset$, there exists $g \in \Gamma$ such that $g(U_1) \cap U_3 = \emptyset$ and $\mathrm{supp}(g) \cap U_2= \emptyset$;
    \item for all $ U_1, U_2 \in \mathcal{U }$, there exists $U_3 \in \mathcal{U}$ such that $U_1 \cup U_2 \subseteq U_3 $.\end{enumerate}
\end{definition}
\begin{theorem}[\cite{gardella2023generalisations} Theorem A]
    Let $\G$ be an essentially principal, ample groupoid. Then the following are equivalent:
    \begin{itemize}
        \item $\G$ is purely infinite and minimal (as in Definition \ref{purely infinite})
        \item $\mathsf{D}(\G)$ is vigorous.
        \item The action of the subgroup of $\mathsf{D}(\G)$ that stabilises a point $x_0$ on $\G^{(0)}\setminus \{x_0\}$ is compressible.
    \end{itemize}
    \label{tgardellathm}
\end{theorem}

\section{Stein's Groups}
We now turn to the piecewise linear homeomorphisms perspective of Thompson-like groups, following Stein \cite{stein1992groups}. Note similar generalisations of Thompson's group $F$ have been studied in detail by Bieri-Streibel \cite{BieriStrebel}. 

\begin{definition}[Stein's Groups]
Let $\Lambda \subset \mathbb{R}$ be a multiplicative subgroup of $(0,+ \infty)$, $\Gamma$ be a $\mathbb{Z} \cdot \Lambda$ submodule such that $\Lambda \cdot \Gamma=\Gamma$ let $1 \leq \ell \in \Gamma$. Then, $V(\Gamma,\Lambda, \ell)$ denotes the right continious, piecewise linear bijections of $[0,\ell]$ with slopes in $\Lambda$ and finitely many discontinuities in $\Gamma$. In addition, let us define two nested subgroups in analogy of Thompson's groups $F \subset T \subset V$:
\begin{itemize}
    \item $T(\Gamma,\Lambda,\ell) \subset V(\Gamma,\Lambda,\ell)$ is the space of  continious piecewise linear bijections of $[0,\ell]/\sim$ where $0 \sim \ell$,  with nondifferentiable points in $\Gamma$ and slopes in $\Lambda$.
    \item $F(\Gamma,\Lambda, \ell) \subset T(\Gamma,\Lambda,\ell) \subset V(\Gamma,\Lambda,\ell)$ is the space of piecewise linear homeomorphisms of $[0,\ell]$ with nondifferentiable points in $\Gamma$ and slopes in $\Lambda$. Note then necessarily, $\forall f \in F(\Gamma,\Lambda, \ell), f(0)=0 \equiv f(1)=1 (\text{mod} \, \mathbb{Z})$. In particular then, $F(\Gamma,\Lambda, \ell) \subset T(\Gamma,\Lambda, \ell)$ 
\end{itemize}\label{steins group}
\end{definition}
Note that the restriction of $\ell \geq 1$ does not limit the diversity of these groups. This definition due to Stein encompasses many generalisations of Higman-Thompson groups of interest to geometric Group Theory.

Occasionally, it is convenient to require our Stein's groups to act on noncompact intervals, taking for example piecewise linear bijections on $\mathbb{R}$, or on $[0,+\infty)$ as in \cite{BieriStrebel}, \cite{stein1992groups}.
\begin{definition}[Noncompact $V(\Gamma,\Lambda,U)$]
  Let $\Lambda$ be a multiplicative subgroup of $\mathbb{R} \cap (0,\infty)$, and $\Gamma$ be a $\mathbb{Z} \cdot \Lambda$ submodule. Let $U$ be some closed, not necessarily compact, interval in $\Gamma$ (for example $[0,\ell]$, $(-\infty,0]$, $[0,\infty)$, $\mathbb{R}$). Then, let $V(\Gamma,\Lambda,U)$ denote the group of piecewise linear bijections of $U$ with finitely many slopes (all in $\Lambda$) and finitely many nondifferentiable points, (all in $\Gamma$).
  \label{noncompact}
\end{definition}

We list some specific examples below:
\begin{example}[Thompson's group $V$]
Let $\Gamma=\mathbb{Z}[1/2], \Lambda=\langle 2 \rangle$ and $\ell=1$. In this notation, $V( \mathbb{Z}[1/2], \langle 2 \rangle,1 )$ recovers Thompson's group $V$ \cite{thompson}, \cite{stein1992groups}. \label{thompsons group}
\end{example}
This group was the first example of a simple, finitely presented group \cite{thompson}. Since then, it has been shown that $V$ is type $F_\infty$ \cite{brown1987finiteness} (see \cite{thumann2014topological} for a very general approach) and is acyclic \cite{szymik2019homology}. These groups were generalised by Higman.  
\begin{example}[The Higman-Thompson groups] Let $n,r \in \mathbb{N}$. Then,
$V(\mathbb{Z}[1/n], \langle n \rangle , r) \cong V_{n,r}$, recovers the Higman-Thompson groups \cite{stein1992groups}, \cite{higman1974finitely}. These are the groups of piecewise linear bijections of $[0,r]$ with slopes in $\langle n \rangle$, and finitely many nondifferentiable points, all of which in $\mathbb{Z}[1/n]$. 
\label{higmans group}
\end{example}
Once again, these groups are type $F_\infty$. A full classification of Higman-Thompson groups has been obtained. Here, we can see the dependence on the length $\ell$ of the underlying interval. In this notation, the main result of \cite{pardo2011isomorphism} states that:
$$V(\mathbb{Z}[1/n],\langle n \rangle , r) \cong V(\mathbb{Z}[1/n'],\langle n' \rangle , r') \iff n=n', \text{gcd}(n-1,r)=\text{gcd}(n-1,r') $$ 
As well as including integers, one can consider arbitrary irrational numbers. The most studied group of this form is Cleary's group \cite{cleary2000regular} \cite{Clearymore} \cite{irrationalslope22} \cite{burillo2021irrationalslope}:
\begin{example}[Cleary's irrational slope Thompson Group]
Let $\tau=\frac{1+ \sqrt{5}}{2}$ be the golden ratio. Let  $V(\mathbb{Z}[\tau,\tau^{-1}], \langle \tau \rangle ,1)$, recovers Clearys group $V_\tau$,  otherwise known as the irrational slope Thompsons group.
\label{clearys group}
\end{example}
Much is known about Cleary's group $V_\tau$, and there have been exciting developments in recent years.  For example, explicit finite generating sets have been found \cite{irrationalslope22}. The abelianisation of this group was shown to be $\mathbb{Z}_2$ \cite{irrationalslope22}. However, it can be shown that it has very different properties to that Thompson's group $V$ as well.

\begin{rmk}
  
From a dynamical perspective, $V_\tau$ contains minimal homeomorphisms of the Cantor space, a concrete example being the map $$f_\tau: [0,1] \rightarrow [0,1] \quad t \mapsto t+ \tau \; (\text{ mod } \mathbb{Z})$$ Thompson's group $V$ does not contain minimal homeomorphisms, since if we examine the orbit structure of an element of Thompson's group $V$, we will find that there are finite orbits and the underlying homeomorphism of the Cantor space is therefore not minimal \cite{salazar2010thompson}. 
  \label{minimal discussion}
\end{rmk}
Finally, it is important to consider when $\Gamma,\Lambda$ are generated by more than one number, i.e. when $\Lambda \not \cong \mathbb{Z}$. The case where $\Lambda$ is generated by multiple integers were the main object of study by Stein, and so are accordingly named Stein Integral groups:
\begin{example}[Stein's Integral Groups]
Let $N=\{n_1,n_2,...n_k\}$ be a finite set of algebraically independent integers. Let $\ell \in \mathbb{N}$. Then, we have that Stein's integral group are realised in this matter as $V(\mathbb{Z}[\prod_{i=1}^k 1/n_i ],\langle n_i \rangle_{i=1}^k, \ell)$.
\end{example}
Stein's integral groups are of type $F_\infty$, and their abelianisation was computed in Stein's original paper \cite{stein1992groups}. But the class of Stein's group studied in this paper is much more general than any of the above example classes. Let us demonstrate that whichever group $\Lambda$ and module $\Gamma$ we choose, we obtain nontrivial groups.
\begin{lemma}
Let $\Lambda$ be a multiplicative subgroup of $\mathbb{R} \cap (0,+\infty)$. Let $\Gamma$ be a $\mathbb{Z} \cdot \Lambda$ submodule and $\ell \in \Gamma$. Then, all three of the groups $F(\Gamma,\Lambda,\ell) \subset T(\Gamma,\Lambda, \ell) \subset V(\Gamma,\Lambda, \ell)$ are nontrivial, properly nested groups. 
\end{lemma}
\begin{proof}

This proof is entirely constructive. We show that for $\ell=1$, since the other cases are analagous. For any choice of $\Lambda,\Gamma$, consider $\lambda \in \Lambda, \lambda < 1/2$ then consider the following piecewise linear functions:

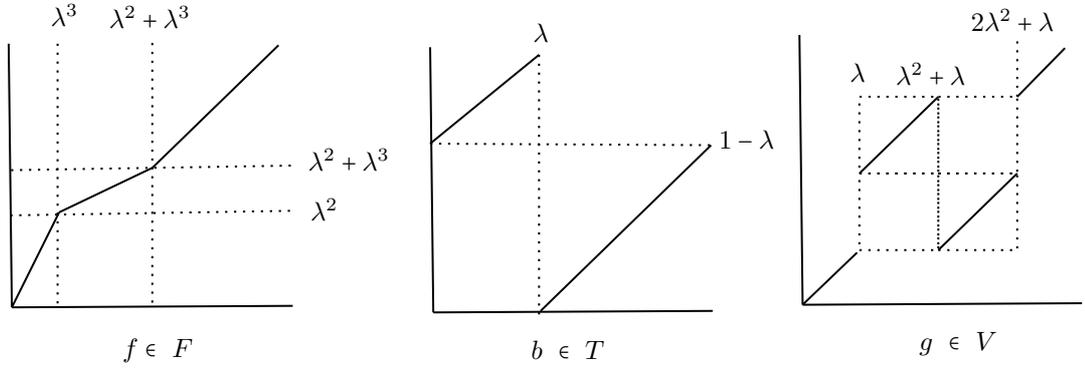
\begin{center}

\tikzset{every picture/.style={line width=0.75pt}}%set default line width to 0.75pt        

\begin{tikzpicture}[x=0.75pt,y=0.75pt,yscale=-0.85,xscale=0.85]
%uncomment if require: \path (0,342); %set diagram left start at 0, and has height of 342

%Straight Lines [id:da23792128113499822] 
\draw    (14,106.4) -- (15.86,262.43) ;
%Straight Lines [id:da8579466500698678] 
\draw    (15.86,262.43) -- (180.75,261.55) ;
%Straight Lines [id:da9365047414508711] 
\draw    (261.4,108.49) -- (263.25,265.43) ;
%Straight Lines [id:da24909226912585858] 
\draw    (263.25,265.43) -- (427.22,264.54) ;
%Straight Lines [id:da39100099778943176] 
\draw    (478.17,101.24) -- (480.03,260.9) ;
%Straight Lines [id:da2868113826782772] 
\draw    (480.03,260.9) -- (644,259.99) ;
%Straight Lines [id:da5299486881347149] 
\draw    (261.4,165.56) -- (325.32,112.95) ;
%Straight Lines [id:da08672123359780715] 
\draw    (326.24,264.54) -- (426.3,166.45) ;
%Straight Lines [id:da8172003546650504] 
\draw    (98.31,179.88) -- (172.41,107.31) ;
%Straight Lines [id:da968112128099778] 
\draw    (15.86,262.43) -- (43.64,206.19) ;
%Straight Lines [id:da4026269904933004] 
\draw    (98.31,179.88) -- (43.64,206.19) ;
%Straight Lines [id:da6023036508890438] 
\draw  [dash pattern={on 0.84pt off 2.51pt}]  (42.72,106.4) -- (42.72,262.43) ;
%Straight Lines [id:da6186575193294861] 
\draw  [dash pattern={on 0.84pt off 2.51pt}]  (98.31,105.96) -- (98.31,261.99) ;
%Straight Lines [id:da8798531387307178] 
\draw  [dash pattern={on 0.84pt off 2.51pt}]  (325.32,112.95) -- (325.32,268.99) ;
%Straight Lines [id:da9687491508737376] 
\draw  [dash pattern={on 0.84pt off 2.51pt}]  (14.93,208) -- (180.75,205.28) ;
%Straight Lines [id:da7530050442753493] 
\draw  [dash pattern={on 0.84pt off 2.51pt}]  (15.4,181.24) -- (181.23,178.52) ;
%Straight Lines [id:da13387953546469356] 
\draw  [dash pattern={on 0.84pt off 2.51pt}]  (261.4,165.56) -- (426.3,166.45) ;
%Shape: Rectangle [id:dp8777551806997304] 
\draw  [dash pattern={on 0.84pt off 2.51pt}] (513.35,183.18) -- (559.67,183.18) -- (559.67,228.54) -- (513.35,228.54) -- cycle ;
%Shape: Rectangle [id:dp05083448347066377] 
\draw  [dash pattern={on 0.84pt off 2.51pt}] (559.67,183.18) -- (605.99,183.18) -- (605.99,228.54) -- (559.67,228.54) -- cycle ;
%Shape: Rectangle [id:dp6095560550854895] 
\draw  [dash pattern={on 0.84pt off 2.51pt}] (513.35,137.82) -- (559.67,137.82) -- (559.67,183.18) -- (513.35,183.18) -- cycle ;
%Shape: Rectangle [id:dp24253446193989348] 
\draw  [dash pattern={on 0.84pt off 2.51pt}] (559.67,137.82) -- (605.99,137.82) -- (605.99,183.18) -- (559.67,183.18) -- cycle ;
%Straight Lines [id:da7897380212953022] 
\draw    (480.03,260.9) -- (512,230) ;
%Straight Lines [id:da4407800260951973] 
\draw    (513.35,183.18) -- (559.67,137.82) ;
%Straight Lines [id:da10668472980481947] 
\draw    (559.67,228.54) -- (605.99,183.18) ;
%Straight Lines [id:da9321404405198075] 
\draw    (605.99,137.82) -- (634,109) ;
%Straight Lines [id:da7367249361204977] 
\draw  [dash pattern={on 0.84pt off 2.51pt}]  (605.99,137.82) -- (606,105) ;

% Text Node
\draw (79.04,279.33) node [anchor=north west][inner sep=0.75pt]    {$f\in \ F$};
% Text Node
\draw (318.92,280.51) node [anchor=north west][inner sep=0.75pt]    {$b\ \in \ T$};
% Text Node
\draw (546.7,275.07) node [anchor=north west][inner sep=0.75pt]    {$g\ \in \ V$};
% Text Node
\draw (189.42,196.38) node [anchor=north west][inner sep=0.75pt]    {$\lambda ^{2}$};
% Text Node
\draw (71.19,82.38) node [anchor=north west][inner sep=0.75pt]    {$\lambda ^{2} +\lambda ^{3}$};
% Text Node
\draw (188.04,166.77) node [anchor=north west][inner sep=0.75pt]    {$\lambda ^{2} +\lambda ^{3}$};
% Text Node
\draw (37.04,80.77) node [anchor=north west][inner sep=0.75pt]    {$\lambda ^{3}$};
% Text Node
\draw (320.04,92.77) node [anchor=north west][inner sep=0.75pt]    {$\lambda $};
% Text Node
\draw (429.04,155.77) node [anchor=north west][inner sep=0.75pt]    {$1-\lambda $};
% Text Node
\draw (506,116.4) node [anchor=north west][inner sep=0.75pt]    {$\lambda $};
% Text Node
\draw (533.04,115.77) node [anchor=north west][inner sep=0.75pt]    {$\lambda ^{2} +\lambda $};
% Text Node
\draw (577.04,83.77) node [anchor=north west][inner sep=0.75pt]    {$2\lambda ^{2} +\lambda $};

\end{tikzpicture}

\end{center}
$$f(t):=\begin{cases} \lambda^{-1}t & t \in [0,\lambda^2] \\
\lambda t + \lambda^2-\lambda^3 & t \in (\lambda^2,\lambda^2 +\lambda^3] \\
t & \text{otherwise}
\end{cases}  $$

$$b(t):=\begin{cases} t+ 1-\lambda & t \in [0,\lambda] \\
 t -\lambda & t \in (\lambda,1] 
\end{cases} \quad g(t):=\begin{cases} t+\lambda^2 & t \in (\lambda,\lambda+\lambda^2] \\
t -\lambda^2 & t \in (\lambda+\lambda^2,\lambda+2\lambda^2] \\
t & \text{otherwise}
\end{cases} $$

These clearly define nontrivial elements of $V(\Gamma,\Lambda,1)$. Furthermore, we have that $f \in F(\Gamma,\Lambda,1),$ $b \in T(\Gamma,\Lambda,1) \setminus F(\Gamma,\Lambda,1), \; \text{ and } g \in V(\Gamma,\Lambda,1) \setminus T(\Gamma,\Lambda,1)$.

\end{proof}
Remark that the existence of nontrivial elements of $F$ also follows from [\cite{BieriStrebel}, Theorem 1]. Let us end by explaining through Cleary's group why there is good reason to think these generalisations are different in an interesting way from Thompson's group $V$. 
In \cite{myintervalexchanges}, the author explored groups of so-called interval exchanges. 
\begin{definition}
Let $\Gamma$ be a countable subadditive group of $\mathbb{R}$, and $\ell \in \Gamma$. Then, $IE(\Gamma,\ell)$ is the group of piecewise linear bijections $f$ of $[0,\ell]$ with finitely many angles $\{ft-t \; : \; t \in [0,\ell] \}$ all in $\Gamma$. 
\end{definition}
Let us explain the relationship between $V(\Gamma,\Lambda,\ell)$ and $IE(\Gamma,\ell)$. One perspective on these groups is that you could consider these groups to be the measure-preserving subgroups of $V(\Gamma,\Lambda,\ell)$ for their actions on $[0,\ell]$. It is notable also, that on a groupoid level, there is an action of $\Lambda$ on the groupoid model for $IE(\Gamma,\ell)$ which gives us a groupoid model for $V(\Gamma,\Lambda,\ell)$. The algebraic relationship between these $V$-like groups and their measure-preserving subgroups is formalised through a Zappa-Szep product decomposition.
\begin{lemma}
    Let $\Lambda$ be a multiplicative subgroup of $\mathbb{R} \cap (0,+\infty)$. Let $\Gamma$ be a $\mathbb{Z} \cdot \Lambda$ submodule and $\ell \in \Gamma$. Then:
    $$V(\Gamma,\Lambda,\ell)=F(\Gamma,\Lambda,\ell) \bowtie IE(\Gamma,\ell) $$
\end{lemma}
\begin{proof}
    We show that $\cdot :IE(\Gamma,\ell) \times F(\Gamma,\Lambda,\ell) \rightarrow V(\Gamma,\Lambda,\ell) \quad  (g,h) \mapsto g \cdot h $ is bijective. Let $f \in V(\Gamma,\Lambda,\ell)$. Then there exists a finite partition of $[0,\ell]$, $[x_0,x_1],(x_1,x_2],$ $...(x_{n-1},x_n] $ such that on each $I_i= (x_i,x_{i+1}]$, $f(t)=\mu_i t + c_i$ where $\mu_i \in \Lambda, c_i \in \Gamma$. Let us construct an element of $\hat{f} \in   F(\Gamma,\Lambda,\ell)$ as follows. On $I_0$  let $\hat{f}(t)=\mu_1 t$ then on $I_i, i>0$, let $\hat{f}(t)=\mu_i t -\mu_i x_i + \sum_{k=0}^{i-1} \mu_k(x_{k+1}-x_{k})$, then let the element of $IE(\Gamma,\ell)$ be $\tilde{f}(t)=t+\mu_i x_i-\sum_{k=0}^{i-1} \mu_k(x_{k+1}-x_{k})+c_i$ for $ t \in I_i$. It is clear that $f=\tilde{f}\hat{f}$ and that this decomposition is unique- one would need an element of $F(\Gamma,\Lambda,\ell)$ with the same derivative on each $I_i$ as $f$, and $\hat{f}$ is the unique group element with this property. 
\end{proof}
 Using this observation, we can see that there are embeddings of interesting groups into Cleary's group $V_\tau$ which would not be possible in $V$:
\begin{rmk}[Juschenko-Monod groups in Clearys group]
One can see that in Cleary's group $V_\tau$, we have a canonical embedding the interval exchange group $IE(\mathbb{Z} \oplus \tau \mathbb{Z}) \hookrightarrow V_\tau$ as described in \cite{myintervalexchanges}. The derived subgroup of $IE(\mathbb{Z} \oplus \tau \mathbb{Z})$ is known, due to results by Juschenko-Monod \cite{juschenkomonod} to be a (rare) example of an infinite, finitely generated, amenable simple group. This shows that the word problem is solvable for $D(IE(\mathbb{Z} \oplus \tau \mathbb{Z}))$, since $D(V_\tau)$ is simple and finitely presented \label{juschenkomonod groups}.
\end{rmk}

Thompson's group $V$ does however embed into Clearys group $V_\tau$, as shown in \cite{irrationalslope22}, \cite{gardella2023generalisations}. Therefore, $V_\tau$ contains a richer class of subgroups than Thompson's group $V$, making us interested in this group from the perspective of embedding questions like the Boone-Higman conjecture. It would be interesting to understand better the  groups that embed into $V_\tau$ but not $V$. Throughout this text, we study $V(\Gamma,\Lambda,\ell)$, with the perspective that this group is the topological full group of a certain ample groupoid. 

\section{Construction of Stein's Groups as Topological Full Groups}
 Let us first adapt $\mathbb{R}$ so that we can allow discontinuities at some subset $\Gamma$, by including two points $\tau_+,\tau_-$, separated in the topology, at each point $\tau \in \Gamma$. This notation is following that in Section 3 \cite{chornyi2020topological}. 
\begin{definition}
Let $\Gamma \subset \mathbb{R}$. Let 
$\mathbb{R}_\Gamma:=\{t, a_+,a_- \; : \; t \in \mathbb{R} \setminus \Gamma, a \in \Gamma \} $
with the canonical quotient map $q$ onto $\mathbb{R}$:
$q: \mathbb{R}_\Gamma \rightarrow \mathbb{R} \quad t \mapsto t, a_{\pm} \mapsto a.$
Let us also define a total order on $\mathbb{R}_\Gamma$ by 
$q(x)<q(y) \implies x < y \; \forall x,y \in \mathbb{R}_\Gamma, a_-<a_+ \forall a \in \Gamma$
And let us topologise $\mathbb{R}_\Gamma$ by the order topology, i.e. the topology generated by open intervals:
$$ (x,y)=\{z \in \mathbb{R}_\Gamma \; : \; x<z<y \}, \; x,y \in \mathbb{R}_\Gamma$$
\end{definition}
In the case of Stein's groups, $\Gamma$ is dense, countable in $\mathbb{R}$. This makes the topology on $\mathbb{R}_\Gamma$ identifiable with the disjoint union of countably many Cantor spaces. 
\begin{lemma}
    Let $\Gamma \subset \mathbb{R}$ be dense and countable. Then a (countable) basis for the topology on $\mathbb{R}_\Gamma$ is given by:
    $$(a_-,b_+)=[a_+,b_-] \; a< b, \; a,b \in \Gamma $$
    Moreover, each set of the form $[a_+,b_-]$ with $a<b$ is a Cantor set. 
\end{lemma}
\begin{proof}
First let us remark that for all $a<b$ that $(a_-,b_+)=[a_+,b_-]$, so that each of these sets are clopen. By density of $\Gamma$, these form a basis for the topology on $\mathbb{R}_\Gamma$. Thus we establish that $\mathbb{R}_\Gamma$ is second countable, with a basis of clopen sets. Note moreover that the basis elements clearly separate points in $\mathbb{R}_\Gamma$. Note that $q$ is continuous, indeed the preimage of $(a,b) \subset \mathbb{R}$ is of the form $(x,y)$ for some $x,y \in \mathbb{R}_\Gamma$. But since $q(a_-,b_+)=q([a_+,b_-])=[a,b]$, compactness follows. In total then we have that for all $a,b \in \Gamma$ $(a_-,b_+)$ is compact, and has a countable basis of compact open subsets; by Brouwer's theorem, we are done. 
\end{proof}
Another perspective of these groupoids, given by Li in \cite{xinlambda}, is that they are the groupoids associated to the inverse semigroups $\Gamma \cap [0,+\infty) \ltimes \Lambda$, i.e. the positive cone of $\Gamma \ltimes \Lambda$. Let us explain the relationship between the space $\mathbb{R}_\Gamma$ and the space of characters here. 
\begin{rmk}[Relationship between $\mathbb{R}_\Gamma, O_{\Gamma^+ \subseteq \Gamma}$]
In \cite{xinlambda}, Section 2.3 Li constructs an analagous space as follows. Let $\Gamma$ be an additive subgroup of $\mathbb{R}$. Let $D(\Gamma^+)$ be the (abelian) semigroup C*-algebra of $\Gamma \cap [0,\infty)$. This group has the basis of idempotents $\{ 1_{a + \Gamma^+ } \; a \in \Gamma \cap [0,+\infty)\} $ He then considers the Gelfand dual space $\Omega(D(\Gamma^+))$, and removes the trivial character $\chi_\infty$ such that for all $a \in \Gamma$ $\chi_\infty(1_{a+\Gamma^+})=1$. This space is denoted $O_{\Gamma^+ \subseteq \Gamma}$ Concretely, this is the space of nonzero, nontrivial, characters
$ \chi: D(\Gamma^+) \rightarrow \{0,1\}$ that are strictly decreasing with on the basis (with respect to the partial order $1_{a + \Gamma^+} \leq 1_{b+ \Gamma^+ } \iff a \leq b$).  This space is topologised in the weak operator topology, which is the topology generated by the basic compact open sets $U_{a,b}:=\{\chi \in O_{\Gamma^+ \subseteq \Gamma} \; : \; \chi(1_{a+\Gamma^+})=1, \chi(1_{b+\Gamma^+})=0 \}, a,b \in \Gamma$. There is a canonical homeomorphism:
$$ f: \mathbb{R}_\Gamma \rightarrow O_{\Gamma^+ \subseteq \Gamma} \quad a_+ \mapsto \chi_a^+, \; a_- \mapsto \chi_a^-, \; t \mapsto \chi_t \quad a \in \Gamma, t \in \mathbb{R}_\Gamma \setminus \Gamma_\pm  $$
Where for all $a,b \in \Gamma$, $t \in \mathbb{R}_\Gamma \setminus \Gamma_\pm  $,
$ \chi_a^+(1_{b+\Gamma^+})=1 \iff b\leq a, \;  \chi_a^-(1_{b+\Gamma^+})=1 \iff b<a, \; \chi_t(1_{b+\Gamma^+})=1 \iff b<t $. 
For all $a,b \in \Gamma, a<b$ $f[a_+,b_-])=U_{a,b}$. \label{identification of the spaces}
\end{rmk}

Let $\ell \in \Gamma \cap (0,+\infty)$. Then, let $q^*$ define a canonical inclusions of $[0,\ell]$ into $[0_+,\ell_-]$ that sends $(a,b] \mapsto (a_+,b_-]$ for all $a,b \in \Gamma$ with $a<b$:
$$q^*:[0,\ell] \hookrightarrow [0_+,\ell_-] \quad t \mapsto \begin{cases} t & t \nin \Gamma \ \\
0_+ & t=0 \ \\
t_- & t \in \Gamma \setminus \{0\}
\end{cases}$$
Let $\Lambda$ be a multiplicative subgroup of $(0,+\infty)$ and $\Gamma$ a $\mathbb{Z} \cdot \Lambda$-submodule, let us form the semidirect product in the following way. Let $\Lambda \acts \Gamma$ by multiplication, and for us to form the semidirect product:
$$ (a,\mu)(b,\lambda)=(\lambda^{-1}a+b,\mu \lambda) \quad \forall (a,\mu),(b,\mu) \in \Gamma \ltimes \Lambda$$
Note $(a,\mu)^{-1}=(-\mu a,\mu^{-1}), \, \forall (a,\mu) \in \Gamma \ltimes \Lambda$. Then, using the canonical action of $\Gamma \ltimes \Lambda$ on $\mathbb{R}_\Gamma$, we can define a partial action that realises $V(\Gamma,\Lambda,\ell)$ as a topological full group.
\begin{lemma}[$V(\Gamma,\Lambda,\ell)$ as the topological full group of a partial action] \label{v as a topological full group}
Let $\Lambda $ be a countable submultiplicative group of $(0,+\infty)$, and $\Gamma$ be a $\mathbb{Z} \cdot \Lambda$-submodule. Let $1<\ell \in \Gamma$. Consider the following canonical action of $\Gamma \ltimes \Lambda$ on $\mathbb{R}_\Gamma$:
$$\beta:\Gamma \ltimes \Lambda  \acts \mathbb{R}_\Gamma $$
Where for all $a_\pm \in q^{-1}(\Gamma) \cap \mathbb{R}_\Gamma, t \in q^{-1}(\Gamma^c) \cap\mathbb{R}_\Gamma$ and all $(c,\mu) \in \Gamma \ltimes \Lambda$
$$ (c,\mu)(a_\pm)=(\mu (a+c))_\pm \quad (c,\mu)(t)=\mu (t+c) $$
Consider the restriction of $\beta$ to a partial action on $[0_+,\ell_-]$. We have that:
$$ \mathsf{F}(\Gamma \ltimes \Lambda \ltimes [0_+,\ell_-]) \cong V(\Gamma,\Lambda,\ell)$$
\end{lemma}
\begin{proof} 
First let us note that $\beta$ does define an action on $\mathbb{R}_\Gamma$. For all $t \in \mathbb{R}_\Gamma \cap q^{-1}(\Gamma^c)$ 
$$ (c,\mu)((b,\nu)(t))=(c,\mu )(\nu t + b)=\mu((\nu t + b)+c)=\mu\nu (t+b + \nu^{-1}c )=((c,\mu)(b,\nu))(t)$$
And similarly for $a_\pm \in \mathbb{R}_\Gamma \cap q^{-1}(\Gamma) $.
Now let us follow Example \ref{description of a tfg of a partial action} for understanding the ample groupoid.
 A basis of the compact open bisections of $\Gamma \ltimes \Lambda \ltimes [0_+,\ell_-]$ is given by 
$$ ((c,\mu),[a_+,b_-]) \quad c \in \Gamma \quad a,b \in \Gamma \quad \max\{-c,0\} \leq a < b \leq \min\{\mu^{-1}-c,1\}  $$
Hence elements of $\mathsf{F}(\Gamma \ltimes \Lambda \ltimes [0_+,\ell_-])$ are homeomorphisms $f$ of $[0_+,\ell_-]$ for which there exists a finite subset $\{x_i\}_{i=1}^n \subset \Gamma$ with $0=x_1<...<x_n=1$ and elements $\{(c_i,\mu_i) \}_{i=1}^n$ such that $$f|_{[(x_i)_+,(x_{i+1})_-]}=\beta(c_i,\mu_i)$$
From here, the proof follows since
$$\varphi: \mathsf{F}(\Gamma \ltimes \Lambda \ltimes_{\alpha} [0_+,1_-])  \rightarrow V(\Gamma,\Lambda,\ell) \quad g \mapsto qgq^* $$
provides an explicit isomorphism. 
\end{proof}

Let us also remark that $[0_+,\ell_-]$ is always $\Gamma \ltimes \mathbb{R}_\Gamma$ full (in fact, it meets every $\Gamma$ orbit in $\mathbb{R}_\Gamma$, by density, since for all $x \in \mathbb{R}_\Gamma$, we can choose $a \in \Gamma$ such that $0<q(x)-a<\ell$). 
It is time to remark on some basic facts for these partial action groupoids:

\begin{lemma}
    $\beta:\Gamma \ltimes \Lambda \acts \mathbb{R}_\Gamma$ is a topologically free, amenable, minimal action. In other words, as a groupoid, $\Gamma \ltimes \Lambda \ltimes_\beta [0_+,\ell_-]$ is topologically principal, amenable and minimal.  \label{groupoid minimal} 
\end{lemma}
\begin{proof}
The action $\beta$ is amenable since $\Gamma \ltimes \Lambda$ is an amenable group. 

Now let us establish minimality. First let us remark the following convergence rules in $\mathbb{R}_\Gamma$:
$$\lim_{n \rightarrow \infty} x_n=x_+ \iff \lim_{n \rightarrow \infty} q(x_n)=q(x) \text{ from above }$$
$$\lim_{n \rightarrow \infty } x_n=x_- \iff lim_{n \rightarrow \infty} q(x_n)=q(x) \text{ from below }$$
$$ \lim_{n \rightarrow \infty } x_n=x \text{ s.t. } x \nin \Gamma \iff \lim_{n \rightarrow \infty } q(x_n)=x$$
For all $x \in \mathbb{R}_\Gamma$, the image of the orbit $\Gamma x$ under $q$, $q(\Gamma x)=q(x)+ \Gamma \subset \mathbb{R}$ is both left-dense and right-dense in $\mathbb{R}$, so we can find sequences in $q(x)+\Gamma$ tending to any $x' \in \Gamma$ from above or below, and sequences approximating any $x' \nin \Gamma$ in $q(\Gamma x)$. Minimality follows. 

It remains to show that $\beta$ is topologically free.
Let $(c,\mu) \in \Gamma \ltimes \Lambda$, with $(c,\mu) \neq (0,1)$. Suppose that $(c,\mu)x=x$ for $x \in \mathbb{R}_\Gamma$. This occurs $\iff q((c,\mu)x)=q(x) \iff q(x)(1-\mu)=c\mu$. Note then $\mu \neq 1$, otherwise this forces $c=0$. Otherwise, we have that $(c,\mu)$ fixes only $q^{-1}(c\mu/(1-\mu))=\{c\mu/(1-\mu)_\pm \}$. In particular, $(c,\mu)$ does not fix $x$ such that $q(x) \nin \Gamma$, so $\mathbb{R}_\Gamma \setminus \{x_\pm \; : \; x \in \Gamma \}$ is a dense set which $\Gamma \ltimes \Lambda$ acts freely on. 
\end{proof} 

Note in particular as a Corollary of Matui's isomorphism Theorem (Theorem \ref{matui isomorphism theorem}), and the regularity established for $\beta$ are conjugate to some known groupoids. Let $\mathcal{E}_k$ denote the full shift on $k$ generators, and $\mathcal{R}_r$ be the full equivalence relation on $r$ points. We have the following:
\begin{corollary}
    Up to groupoid conjugacy, $\Gamma \ltimes \Lambda \ltimes_\beta [0_+,\ell_-]$ is the unique ample groupoid of germs $\G$ such that $\mathsf{F}(\G) \cong V(\Gamma,\Lambda,\ell)$. In particular, for all $2<k \in \mathbb{N}$, $r \in \mathbb{N}$, 
$\mathbb{Z}[1/k] \ltimes \langle k \rangle \ltimes_\beta [0_+,r_-] \cong \mathcal{R}_r \times \mathcal{E}_k$. \label{unique groupoid conjugacy} 
\end{corollary}
\begin{rmk}
    Following on from Remark \ref{identification of the spaces}, let us make the point that there are equivalent but nonetheless alternative ways to construct this groupoid (up to groupoid conjugacy). One such description arises from cancellative semigroups. Let $\Gamma^+$ denote the positive cone of $\Gamma$. Consider the subsemigroup $\Gamma^+ \ltimes \Lambda$ of $\Gamma \ltimes \Lambda$,  and the universal groupoid (in the sense of Li) $\G(\Gamma^+ \ltimes \Lambda)$ of this semigroup. This groupoid is the semigroup transformation groupoid of the canonical action of the semigroup on the dual of the semilattice of idempotents, which in this case is just the dual of $\Gamma^+$. If we take the full corner of this groupoid restricted to  the subset $U$ of characters $\chi$ such that $\chi(\ell)=1$, this describes a conjugate groupoid. See [\cite{xinlambda}, \cite{cuntz2017k}, Chapter 5] for more information on the universal groupoids of semigroups that embed into groups.  
\end{rmk}

We next verify that the groupoid $(\Gamma \ltimes \Lambda) \ltimes_{\beta} [0_+,\ell_-]$ is purely infinite, in the sense of Definition \ref{purely infinite}.

\begin{lemma}
    Let $\Lambda$ a nontrivial multiplicative subgroup of $\mathbb{R}$, $\Gamma$ a $\mathbb{Z} \cdot \Lambda$-submodule and $\ell \in \Gamma$. Then $(\Gamma \ltimes \Lambda) \ltimes_{\beta} [0_+,\ell_-]$ is a purely infinite groupoid. \label{groupoid purely infinite}
\end{lemma}
\begin{proof}
   
By compactness, any clopen subset can be written as a finite disjoint union of sets of the form $[a_+,b_-]$. Therefore, let $A=\cup_{i=1}^n [(a_i)_+,(b_i)_-]$ where each $a_i,b_i \in \Gamma$. Then let $\lambda \in \Lambda$ with $0 \leq \lambda \leq 1/2$, set
$$U= \bigsqcup_{i=1}^n  ((a_i(1-\lambda^{-1}),\lambda),[(a_i)_+,(b_i)_-]),  \; V=\bigsqcup_{i=1}^n(b_i(1-\lambda^{-1}),\lambda), [(a_i)_+,(b_i)_-]).$$ 
It is straightforward to verify that $U,V$ are indeed bisections.

We compute
$$U^{-1}=\bigsqcup_{i=1}^n ((-\lambda a_i (1-\lambda^{-1}), \lambda) , [(a_i)_+,(a_i+\lambda(b_i-a_i))_-]) $$ $$
V^{-1}=\bigsqcup_{i=1}^n(-\lambda b_i (1-\lambda^{-1}), \lambda),[(b_i-\lambda(a_i-b_i))_+,(b_i)_-]) $$
This gives us exactly what we need. $s(U)=s(V)=\bigsqcup_{i=1}^n[(a_i)_+,(b_i)_-]=A$, $r(U)= \bigsqcup_{i=1}^n [(a_i)_+, (a_i+\lambda (a_i-b_i)_-]$, $r(V)= \bigsqcup_{i=1}^n [(b_i-\lambda(a_i-b_i))_+,(b_i)_-]$ and so $r(V) \cap r(U) = \emptyset$ whilst simultaneously $r(U) \cup r(V) \subset A$. 
\end{proof}
\begin{rmk}
An alternative, (group theoretic) proof that $(\Gamma \ltimes \Lambda) \ltimes_{\beta} [0_+,\ell_-]$ is a purely infinite, minimal groupoid is to observe that $V(\Gamma,\Lambda,\ell)$ is vigorous for all choices of $\Gamma,\Lambda,\ell$ and apply Theorem A \cite{gardella2023generalisations} (Theorem \ref{tgardellathm}). 
\end{rmk}
Hence, by Theorem \ref{d simple} we obtain that $D(V(\Gamma,\Lambda,\ell))$ is always simple. We also have via [\cite{gardella2023generalisations}, Cor D] that  $D(V(\Gamma,\Lambda,\ell))$ has no proper characters.
\begin{corollary}[Simple, Vigorous Derived Subgroup]
Let $\Gamma,\Lambda,\ell$ be arbitrary. Then,   $D(V(\Gamma,\Lambda,\ell))$ is simple and vigorous. Moreover, it is identified with the alternating subgroup $A(\Gamma \ltimes \Lambda \ltimes [0_+,1_-])$. Finally, $D(V(\Gamma,\Lambda,\ell))$ has no proper characters. \label{simple derived subgroup}
\end{corollary}

Using Definition, we can also realize the groups $V(\Gamma,\Lambda,U)$ which act upon on noncompact intervals. 
\begin{lemma}
Let $\Lambda$ be a multiplicative subgroup of $(0,+\infty)$. Let $\Gamma$ be a $\mathbb{Z} \cdot \Lambda$ submodule. Let $U$ be a closed interval of $\mathbb{R}$ with nonempty interior. Then,
$V(\Gamma,\Lambda,U) \cong F( \Gamma \ltimes \Lambda \ltimes U) $. The derived subgroup, $D(V(\Gamma,\Lambda,U))$ is simple and vigorous. 
\end{lemma}

\section{Finite Generation of The Derived Subgroup}

\subsection{Expansivity of Partial Transformation Groupoids}
Nekrashevych showed that the alternating group of a topological full group is finitely generated if the underlying ample groupoid has a technical condition known as expansivity. We recall the definition below:
\begin{definition}[Expansive \cite{nekrashevych2019simple}]
For  an \'etale Cantor groupoid $\G$:
\label{expansive groupoid}
\begin{itemize}
    \item A compact set $K \subset \G$ is called a compact generating set such that $\G= \bigcup_{n \in \mathbb{N}}(K \cup K^{-1})^n$.
    \item A finite cover $\mathcal{B}=\{B_{i}\}_{i=1}^N$ of bisections is called expansive if $\bigcup_{n \in \mathbb{N}} (\mathcal{B} \cup \mathcal{B}^{-1})^n$ forms a basis for the topology of $\G$. 
    \item $\G$ is called expansive if there is a compact generating subset $K$ with an expansive cover $\mathcal{B}$.
    \label{expansive nekrashevych}
\end{itemize}\end{definition}
\begin{theorem}Let $\G$ be an expansive Cantor groupoid with infinite orbits. Then $\mathsf{A}(\G)$ is finitely generated [\cite{nekrashevych2019simple}, Theorem 5.6]. \label{finite generatedness theorem}
\end{theorem}
    
There is also a natural notion of expansivity for partial actions, which generalises the natural notion of expansivity of global actions.  
\begin{definition}[Expansive partial action]
    Let $\alpha:G \acts X$ be a partial action. We say that $\alpha$ is expansive if there exists $\epsilon >0$ such that for all $x,y \in X, \, x \neq y$, there exists $g \in G$ such that $d(g x, g y)>\epsilon$. \label{partial expansive}
\end{definition}
Nekrashevych showed that a transformation groupoid of a finitely generated group acting on the Cantor space is expansive if and only if the underlying action is expansive. We give an analogous result for partial actions.  
\begin{lemma}
Let $\alpha:G \acts X$ be an expansive partial action, of a discrete group on the Cantor space and suppose $G \ltimes_\alpha X$ is compactly generated.
Then, the partial transformation groupoid $G \ltimes_\alpha X$ is expansive in the sense of Definition \ref{expansive nekrashevych}. \label{ expansive partial then expansive}

\end{lemma}
\begin{proof}
Let $\epsilon>0$ be as in Definition \ref{partial expansive}.
    Let $K$ be a compact generating set for the groupoid $G \ltimes_\alpha X$. By ampleness, there exists a finite cover of $K$ by open compact bisections.  Moreover, we may refine this cover to a cover $\mathcal{S}$ of $K$ by bisections such that for each $S \in \mathcal{S}, $ both $r(S)$ and $s(S)$ have diameter less than $\epsilon$. We claim that this cover is expansive. It follows that for all $A \in \bigcup_{n=1}^\infty   (\mathcal{S} \cup \mathcal{S}^{-1})^n$, $r(A)$ and $s(A)$ have diameter less than $\epsilon$ with respect to the metric,  since $r(A'A) \subset r(A')$ and $ s(AA')\subset s(A')$. Note that also for all $A \in \bigcup_{n=1}^\infty   (\mathcal{S} \cup \mathcal{S}^{-1})^n$, we can assume $A=(g, U)$ where $U \subset \G^{(0)}$ and $g \in G$.

Recall [\cite{nekrashevych2019simple}, 5.3]. We have that $\mathcal{S}$ is expansive if and only if for every $x \neq y$, there exists a subset of the form $s(A)$ such that $x \in s(A)$, $y \nin s(A)$ and $A \in (\mathcal{S} \cup \mathcal{S}^{-1})^n$ for some $n$. (This is condition (4)). 

Let $g$ be such that $d(gx,g y)> \epsilon$. For $n$ large enough, by compact generation we have that there exists $A$ such that $(g,x) \in A \in (\mathcal{S} \cup \mathcal{S})^n$. By construction then, $(1,x) \in s(A)$. Suppose $(1,y) \in s(A)$. Then $(g,y) \in A$, $(1,gy) \in r(A)$. But $A$ must have diameter with length less than $\epsilon$ by construction. Therefore, we have $d(gx,gy)< \epsilon$, a contradiction. Therefore $(g,y) \nin A$, and so the underlying groupoid is expansive in the sense of Nekrashevych. 
\end{proof}

In contrast to global actions, understanding when a partial transformation groupoid is compactly generated is a subtle question. For transformation groupoids on compact spaces, $X \ltimes_\alpha G$, compact generation of the groupoid is equivalent to the finite generation of $G$. One does not have an analogous result in the case of partial actions. However, we below show that restricting compactly generated groupoids to certain subsets of the unit space preserves compact generation:
\begin{theorem}
    Let $\G$ be an ample compactly generated groupoid. Let $U$ be a full clopen subset of $\G^{(0)}$ (i.e. $U$ meets every $\G$ orbit). Then $\G|_U^U$ is compactly generated. 
    \label{lemma compact generation full subsets}
\end{theorem}
\begin{proof}
    Let $\mathcal{S}_1=\{B_1,...,B_k\}$ be the compact generating set of $\G$.  For $i=1,...,k$ let:
    $$S_i=B_i|_U^U \quad E_i=B_i|_{U^c}^{U^c} \quad  T_i=B_i|_{U^c}^U \sqcup B_i^{-1}|_{U^c}^U $$
    Then $\mathcal{S}_2=\{S_i,E_i,T_i\}^{k}_{i=1}$ is a compact generating set of $\G$, since for all $0<i \leq k$, $B_i=S_i \sqcup E_i \sqcup T_i|_{U^c} \sqcup T_{i}^{-1} |_{U}$. 

    By fullness, for all $x \in U^c$ there exists a groupoid element $g \in \G$ such that $s(g)=x, r(g) \in U$. By ampleness then, there exists a compact open bisection $F_x$ containing $g$, such that $x \in s(F_x) \subset  U^c$ and $r(F_x) \subset U$. Then $\{s(F_x)\}_{x \in U^c}$ is an open cover of $U^c$. Since $\G$ is compactly generated, $\G^{(0)}$ must be compact. Then $U^c$, a closed subset of $\G^{(0)}$ must be compact. By compactness, we can refine our open cover $s(F_x), x \in U^c$ of $U^c$ to a finite cover. Therefore, there exists a finite  subset $F_j, j=1,...,m$ of the $F_x$ such that
    $\cup_{j=1}^m s(F_j)=U^c, \cup_{j=1}^m r(F_j)\subset U$. Our claim is that the finite collection of compact open bisections given by:
    $$\mathcal{S}_U=\{S_i, T_i F_j^{-1}, F_j E_i^{\pm 1} F_l^{-1} \}_{ \substack{i=1,...,k \ \\
    j,l=1,...,m}} $$
    forms a compact generating set for $\G|_U^U$.

    Let $g \in \G|_{U}^{U}$. Then there exists a finite word $W$ in $\mathcal{S}_2$ such that $g \in W$. Let us examine the form of a word in $S_i, E_i, T_i$ with source and range in $U$. We have that after each $T_i$, we will be followed by a (possibly empty) word of $E_j^{\pm 1}$ followed by a $T_l$. Therefore, let $2n$ be the number of letters in $W$ of the form $T_i, T_i^{-1} i=1,...,k$. By induction on $n$, we will show that $g \in \bigcup_{k \in \mathbb{N}}(\mathcal{S}_{U} \cup \mathcal{S}_{U}^{-1})^k$.

    Base case: If $n=0$, we have that $W$ is a finite word in $\{S_i, S_{i}^{-1}\}_{i=1}^k$, and so we are done, since $S_i \in \mathcal{S}_U$ for all $i$. 

    Inductive step: assume true for $n$, then for $n+1$ we have that we may rewrite $W$ in the form:
    $$W_0 T_j (E_{k_1}^{\pm 1} E_{k_2}^{\pm 1} ... E_{k_n}^{\pm 1}) T_i^{-1} W_n$$
    Where $W_0$ is a word in $\{S_i, S_i^{-1} \}_{i=1}^k$, and $W_n$ is a word in $\mathcal{S}_U$ such that there are $2n$ letters in $\{T_i,T_i^{-1}\}_{i=1}^k$. By inductive hypothesis we have that the groupoid element ${(W_n)}_{s(g)}$ is in some finite word $\tilde{W}_n$ in $\mathcal{S}_U \cup \mathcal{S}_U^{-1}$.
    We have that $r((T_i^{-1} W_n )_s(g)) \in U^c,$ therefore there exists $l_n$ such that $r((T_i^* W_n )_s(g)) \in s(F_{l_n})$. Similarly, for $j=0,...,n-1$ there exists $l_{j}$ such that $r(E_{k_{j}}^{\pm 1}E_{k_{j+1}}^{\pm 1}...E_{k_{n}^{\pm 1})T_i^{-1}W_n)_{s(g)})} \in s(F_{l_{n-j}}) $
    Then we have that:
    $$ g \in W_0 T_j (F_{l_1}^{-1}F_{l_1}) E_{k_1}^{\pm 1} (F_{l_2}^{-1} F_{l_2})E_{k_2}^{\pm 1}...(F_{l_{n-1}}^{-1} F_{l_{n-1}})E_{k_n}^{\pm 1}((F_{l_{n}}^{-1} F_{l_{n}})T_i^{-1} \tilde{W}_n $$
    Which we may rewrite in the form:
    $$W_0 T_j (F_{l_1}^{-1}) (F_{l_1}E_{k_1}^{\pm 1}F_{l_2}^{-1})...(F_{l_{n-1}}E_{k_n}^{\pm 1}F_{l_{n}}^{-1})(T_iF_{l_n}^{-1})^{-1} \tilde{W}_n, $$ a finite word in $\mathcal{S}_U$. 
\end{proof}

Combining this with our expansivity result for partial actions we obtain:
\begin{corollary}
    Let $\alpha:G \acts X$ be an expansive action of a finitely generated group on a compact space $X$. Let $U$ be $\alpha$-full open subset of $X$. Then the partial transformation groupoid $G \ltimes U$ is expansive in the sense of Nekrahsevych. 
\end{corollary}

Compact generation was shown to be a necessary condition for $\mathsf{A}(\G)$ to be finitely generated in \cite{myintervalexchanges}, we include this proof here for completeness.

\begin{lemma}
    Let $\G$ be an essentially principal ample groupoid with infinite orbits. If $\mathsf{A}(\G)$ is finitely generated then $\G$ is compactly generated. \label{a finitely generated implies}
\end{lemma}
\begin{proof}
    Let $g \in \G$. It is enough to show that there exists some bisection $B \in \mathsf{A}(\G)$ such that $g \in B$, since then the generating set of $\mathsf{A}(\G)$ also serves as a compact generating set for $\G$. If $s(g) \neq r(g)$, by ampleness, we have there exists some compact open bisection $\hat{B}_1$ such that $g \in \hat{B}_1$. By restricting the source of $\hat{B}_1$ if necessary, we may assume that $s(\hat{B}_1)$ and $r(\hat{B_1})$ are disjoint and satisfy $s(\hat{B}_1) \bigcup r(\hat{B_1}) \neq \G^{(0)}$. By minimality, we have there exists some $h \in \G$ with $s(h)=r(g)$ and $r(h) \in (s(\hat{B}_1) \bigcup r(\hat{B_1}))^c$ Then by ampleness, there exists some compact open bisection $B_2$ with $h \in B_2$. Again, by restricting $s(B_2)$ if necessary, we can assume $s(B_2) \subset r(\hat{B}_1)$ and $r(B_2) \subset (s(B_2) \bigcup \hat{B}_1^{-1}(s(B_2))^c$. Now let $B_1=\hat{B}_1 |_{\hat{B_1}^{-1} s(B_2)}$. Then 
$$ \gamma_{B_1,B_2} =B_1 \sqcup B_2 \sqcup (B_1B_2)^{-1} \cup (\G^{(0)} \setminus s(B_1) \cup s(B_2) \cup r(B_2))\in \mathsf{A}(\G)$$
satisfies $g \in \gamma_{B_1,B_2} $.

Otherwise $s(g)=r(g)$. Then, there exist $g_1,g_2$ such that $s(g_i) \neq r(g_i)$ but $g_1g_2=g$. Hence by the above argument, there exist $\gamma_1,\gamma_2 \in \mathsf{A}(\G)$ such that $g_1 \in \gamma_1, \gamma_2 \in B_2$ and hence $g=g_1g_2\in \gamma_1 \gamma_2 \in \mathsf{A}(\G)$. 
\end{proof}
An important consequence of this is that Groupoids with noncompact unit spaces cannot give rise to finitely generated alternating or derived subgroups because they act with full support on noncompact intervals; the underlying groupoids are not compactly generated. 
\begin{rmk}
    
Suppose $\G$ is a minimal ample groupoid such that $\mathsf{A}(\G)$ is finitely generated. Then $\G$ is compactly generated by Lemma \ref{a finitely generated implies}. A consequence of this is that $\G^{(0)}$ is compact, since then if $K$ was our compact generating set of our groupoid, $K$ only moves a compact subset of $\G^{(0)}$, contradicting minimality. The same argument holds for $\mathsf{D}(\G)$ since $\mathsf{A}(\G)$ is contained in $\mathsf{D}(\G)$.

\end{rmk}
\subsection{Finite Generation of Stein's Groups}
In the case of our action, the partial action $\beta$ is indeed expansive, as shown below.
\begin{lemma}
Let $\Gamma$, $\Lambda$, $\ell$ be arbitrary. Then $\beta$ is expansive in the sense of Definition \ref{partial expansive} \label{partial action for v is expansive}
\end{lemma}
\begin{proof}
Let $\lambda \in \Gamma \cap (0,1) $. Let $x,x' \in [0_+,\ell_-]$ be distinct. Let $\epsilon =d([0_+,(\lambda)_-],[(\lambda)_- , 1_+ ])>0$. We separate into two cases:
\begin{itemize}
    \item If $q(x) \neq q(x')$, suppose without loss of generality that $q(x)<q(x')$. By density, suppose the difference of $q(x')-q(x)>c>0$ for some $c \in \Gamma$. Also there exists some $c' \in \Gamma$ such that $\lambda -c <q(x)-c'< \lambda$. Therefore we have that $q(c' x) =q(x)-c'<\lambda $ and $q(c'x')=q(x')-c'>q(x)-c'+c>\lambda-c+c=\lambda$. Therefore, $(c',1) \in \Gamma \rtimes \Lambda$ is a group element separating $x,x'$.  
    \item If $q(x)=q(x) \in \Gamma$, then for $c'=q(x)-\lambda$, $(c',1) \in \Gamma \rtimes \Lambda$ is a group element that will separate the two characters into $[0_+,\lambda_-], [\lambda_+, 1_-]$.\end{itemize}\end{proof}
Applying Theorem \ref{lemma compact generation full subsets}, Lemma \ref{ expansive partial then expansive} and Theorem \ref{finite generatedness theorem} from the previous subsection, it is enough for us to consider compact generation in the case when $\ell=1$.
\begin{corollary}
    Let $\Lambda$ be a submultiplicative group of $(0,+\infty)$, $\Gamma$ be a $\mathbb{Z} \cdot \Lambda$-submodule.  The following are equivalent:
    \begin{enumerate}
      \item The groupoid $\Gamma \ltimes \Lambda \ltimes [0_+,1_-]$ is compactly generated. 
        \item The groupoid $\Gamma \ltimes \Lambda \ltimes [0_+,\ell_-]$ is compactly generated for all $\ell \in \Gamma$. 
        \item $D(V(\Gamma,\Lambda,\ell))$ is finitely generated for all $\ell \in \Gamma$. 
        \item $D(V(\Gamma,\Lambda,\ell))$ is 2-generated for all $\ell \in \Gamma$.

    \end{enumerate}
     \label{finitely generated compactly}
\end{corollary}
\begin{proof}
\begin{itemize}
    \item 1. $\implies$ 2.  Suppose $\Gamma \ltimes \Lambda \ltimes [0_+,1_-]$ is compactly generated and let $\ell$ be arbitrary. Let $\mu \in \Lambda$ be such that $\mu > \ell$. Note that $\Gamma \ltimes \Lambda \ltimes [0_+,1_-] \cong \Gamma \ltimes \Lambda \ltimes [0_+,\mu_-]$, therefore $\Gamma \ltimes \Lambda \ltimes [0_+,\mu_-]$ is compactly generated. $[0_+,\ell_-]$ is full and open as a subset of $[0_+,\mu_-]$. Applying Theorem \ref{lemma compact generation full subsets}, we have that $\Gamma \ltimes \Lambda \ltimes [0_+,\ell_-]$ is compactly generated. 
    \item 2. $\implies$ 3. follows directly as a combination of Lemma \ref{ expansive partial then expansive} and Theorem \ref{finite generatedness theorem}. 
    \item 3. $\implies$ 1. Note then in particular, $D(V(\Gamma,\Lambda,1)$ is finitely generated. Hence by Lemma \ref{a finitely generated implies}, $\Gamma \ltimes \Lambda \ltimes [0_+,1_-]$ is compactly generated.
    \item 4. $\iff$ 3. follows from noting that $D(V(\Gamma,\Lambda,\ell)$ is simple and vigorous (Corollary \ref{d simple}), and applying \cite[Theorem 5.15]{vigorous}.
\end{itemize}
  
\end{proof}
These groupoids are not \textit{always} compactly generated. 
\begin{lemma}
    If $\Gamma \ltimes \Lambda$ is not finitely generated (for example, if $\Lambda$ is of infinite rank) then $\Gamma \ltimes \Lambda \ltimes [0_+,1_-]$ is not compactly generated and thus $D(V(\Gamma,\Lambda,\ell))$ is not finitely generated.
    \label{infinite rank then not f.g.}
\end{lemma}
\begin{proof}
Suppose there is a compact generating set for $\Gamma \ltimes \Lambda \ltimes [0_+,1_-]$. Then, this compact generating set may be written as a finite collection of our basis elements for the topology $((c_i,\mu_i), [(a_i)_+, (b_i)_-]), i=1,...,n$ and therefore there is a finite string in $(c_i,\mu_i)$ that gives every group element in the set $M=\{(c,1) \; (0,\mu) \: \mu \in \Lambda \cap [0,1], c \in [0,1] \}$, since there are bisections $((c,1),[0_+,(1-c)_-]), ((0,\mu), [0_+,1_-])$.  But note that $M$ generates $\Gamma \ltimes \Lambda$ as a group. Hence $(c_i,\mu_i)$ also forms a finite generating set of the group $\Gamma \ltimes \Lambda$. 
\end{proof}
For this reason, the question of compact generation is only relevant when $\Lambda$ is polycyclic. Let us begin with the simplest case, namely the case when $\Lambda$ is cyclic. Here we show explicitly that the groupoid is always compactly generated. Set $\Lambda=\langle \lambda \rangle$, $\Gamma=\mathbb{Z}[\lambda,\lambda^{-1}]$, where $\lambda<1$. 
Let $\G_\lambda=\mathbb{Z}[\lambda,\lambda^{-1}] \ltimes \langle \lambda \rangle \ltimes [0_+,1_-]$. Let $K \in \mathbb{N}$ be large enough that $\lambda+\lambda^{K}<1$. Let $\mu_1= \lambda^{-1}-\lfloor \lambda^{-K-1} \rfloor \lambda^K$:
$$ c=((0,\lambda),[0_+,1_-])$$
$$f_i=((\lambda^i,1), [0_+,(1-\lambda^i)_-]) \sqcup ((\lambda^i-1, 1),[(1-\lambda^i)_+,1_-])$$
$$g_1=((\mu_1,1), [0_+,(1-\mu_1)_-]) \sqcup ((\mu_1-1, 1),[(1-\mu_1)_+,1_-])$$
Let $\mathcal{S}=\{c,f_1,...,f_K, g_1 \}$. Note that $\mathcal{S}$ is compact. Let $\G^{\mathcal{S}}_\lambda$ be the subgroupoid of $\G_\lambda$ generated algebraically by $\mathcal{S}$. We aim to show that $\G_\lambda=\G_\lambda^\mathcal{S}$. Our first claim is that for all $i \in \mathbb{N}$ we have the bisection $f_i$ which adds $\lambda^i ( \text{mod} \mathbb{Z})$.
\begin{lemma} Let $\lambda<1$ be arbitrary, let $\mathcal{S}=\{c,f_1,...,f_K,g_1\}$, and let $\G^\mathcal{S}_\lambda$ be the subgroupoid of $\G_\lambda=\mathbb{Z}[\lambda,\lambda^{-1}] \ltimes \langle \lambda \rangle \ltimes [0_+,1_-]$ generated by $\mathcal{S}$. Then,
    for all $i \in \mathbb{N}$, the bisection
    $f_i=((\lambda^i,1), [0_+,(1-\lambda^i)_-]) \sqcup ((\lambda^i-1, 1),[(1-\lambda^i)_+,1_-]) \in \mathcal{G}^\mathcal{S}_\lambda $\label{fi}
\end{lemma}
\begin{proof}
    We prove this by induction. We know it is true for $i=1,...,K$. Assume true for $i \geq K$, for $i+1$ we have that:
    $$f_{i+1}|_{[0_+,(\lambda-\lambda^{i+1})_-]}=c f_i c^{-1} |_{[0_+,(\lambda-\lambda^{i+1})_-]}. $$
    Let us verify this on each domains:
    $$ [0_+,(\lambda-\lambda^{i+1})_-] \xrightarrow{c^{-1}} [0_+, (1-\lambda^i)_-] \xrightarrow{f_i} [\lambda^i_+, 1_-] \xrightarrow{c} [\lambda^{i+1}_+,\lambda_-].$$
    Similarly, 
    $$f_{i+1}|_{[(\lambda-\lambda^{i+1})_+, \lambda_-]}=f_1c f_i c^{-1} |_{[(\lambda-\lambda^{i+1})_+, \lambda_-]} .$$
     Let us verify this on each domains:
    $$ [(\lambda-\lambda^{i+1})_+, \lambda_-] \xrightarrow{c^{-1}} [(1-\lambda^i)_+,1_-] \xrightarrow{f_i} [0_+,\lambda^i_-] \xrightarrow{c} [0_+,\lambda^{i+1}_-] \xrightarrow{f_1} [\lambda_+,(\lambda+\lambda^{i+1})_-] .$$
    Using that $\lambda^{i+1}<\lambda^i<\lambda^K<1-\lambda$. 
    Taking the union of these two bisections, we have $f_{i+1}|_{[0_+,\lambda_-]}$. In order to obtain $f_{i+1}|_{[0_+,1_-]}$, note that for all $n=0,...,\lfloor \lambda^{-1} \rfloor -1$:
    $$ f_{i+1}|_ {[(n\lambda)_+, ((n+1)\lambda)_-]}=f_1^n f_{i+1}|_{[0_+,\lambda_-]} f_1^{-n}.$$
    Finally, we have that for $n= \lfloor \lambda^{-1} \rfloor $,
    $$ f_{i+1}|_{[(n\lambda)_+, 1_-]}=f_1^n f_{i+1}|_{[0_+,\lambda_-]} f_1^{-n}|_ {[(n\lambda)_+, 1_-]}.$$
    This covers $f_{i+1}$ on $[0_+,1_-]$, completing our inductive step. 
\end{proof}
Our second claim is that for all $i \in \mathbb{N}$ we have the bisection $f_{-i}$ which adds $\lambda^{-i} (\text{mod} \mathbb{Z})$.
\begin{lemma} Let $\lambda<1$ be arbitrary, let $\mathcal{S}=\{c,f_1,...,f_K,g_1\}$, and let $\G^\mathcal{S}_\lambda$ be the subgroupoid of $\G_\lambda=\mathbb{Z}[\lambda,\lambda^{-1}] \ltimes \langle \lambda \rangle \ltimes [0_+,1_-]$ generated by $\mathcal{S}$. Then,
for all $i \in \mathbb{N}$ we have that the bisection
$f_{-i}=((\lambda^{-i}-\lfloor \lambda^{-i} \rfloor,1), [0_+,(1-(\lambda^{-i}-\lfloor \lambda^{-i} \rfloor))_-]) \sqcup ((\lambda^{-i}-\lfloor \lambda^{-i} \rfloor-1, 1),[(1-(\lambda^{-i}-\lfloor \lambda^{-i} \rfloor))_+,1_-]) \in \G^{\mathcal{S}}_\lambda$.\label{gi}\end{lemma}
\begin{proof}

For all $i \in \mathbb{N}$ let (as before) $K$ be large enough that $\lambda+\lambda^K<1$, and set $\mu_{i+1}=\lambda^{-1} \mu_i- \lfloor \lambda^{-1} \mu_i \lambda^{-K} \rfloor \lambda^K$. Our first aim is to show that  for all $i$, we have that the bisection
$g_i=((\mu_i,1), [0_+,(1-\mu_i)_-]) \sqcup ((\mu_i-1, 1),[(1-\mu_i)_+,1_-]) \in \G^{\mathcal{S}}_\lambda$.
  Let us prove the statement by induction. We have that the statement is true by assumption for $i=1$, ($g_1 \in \mathcal{S}$) so let us proceed with the inductive step, assuming true for $i$ and aiming to show it is true for $i+1$. Let $m_i=\lfloor \lambda^{-1} \mu_i \lambda^{-K}\lfloor$. We claim that:
   $$g_{i+1}|_{[0_+,\lambda_-]}=f_K^{-m_i}c^{-1}g_i c |_{[0_+,\lambda_-]}.$$
   Let us verify on our domains, note  $\mu_i<\lambda^K$. Therefore,
$$\lambda^2+\mu_i<\lambda^2+\lambda^K=\lambda(\lambda+\lambda^K)<\lambda.$$
Therefore, $[(\mu_i)_+,(\lambda^2+\mu_i)_-] \subset [0_+,\lambda_-]$. Using this we can verify our domains:
$$ [0_+,\lambda_-] \xrightarrow{c} [0_+,\lambda^2_-] \xrightarrow{g_i} [(\mu_i)_+,(\lambda^2+\mu_i)_-] \xrightarrow{c^{-1}} [(\lambda^{-1} \mu_i)_+, (\lambda+\lambda^{-1} \mu_i)_-] \xrightarrow{f_K^{-m_i}} [(\mu_{i+1})_+, (\lambda+\mu_{i+1})_-].$$
Note that for all $n=0,...,\lfloor \lambda^{-1} \rfloor -1$:
    $$ g_{i+1}|_ {[(n\lambda)_+, ((n+1)\lambda)_-]}=f_1^n g_{i+1}|_{[0_+,\lambda_-]} f_1^{-n}.$$
    Finally, we have that for $n= \lfloor \lambda^{-1} \rfloor $
    $$ g_{i+1}|_{[(n\lambda)_+, 1_-]}=f_1^n g_{i+1}|_{[0_+,\lambda_-]} f_1^{-n}|_ {[(n\lambda)_+, 1_-]}.$$
    This covers $g_{i+1}$ on $[0_+,1_-]$. 

Let us finish by observing that for all $i$ we can recover $f_{-i}$ as a finite word in $g_1,....,g_i$ and $f_j, j \in \mathbb{N}$ which are themselves elements of $\G_\lambda^\mathcal{S}$ by the above argument and Lemma \ref{fi}. This is because $\mu_i$ are polynomials in $\lambda$ of the form $\mu_i=\lambda^{-i}+\sum_{k=1}^N a_k\lambda^{-i+k}$ where $N \in \mathbb{N}$ and $ a_k \in \mathbb{Z}$. 
\end{proof}
Using the above Lemmas, we have all additive germs in $\G_\lambda$ are also in $\G^\mathcal{S}_\lambda$. We use this to conclude that $\G_\lambda=\G^\mathcal{S}_\lambda$. 
\begin{lemma}
Let $\lambda<1$ be arbitrary, let $\mathcal{S}=\{c,f_1,...,f_K,g_1\}$, then, $\G_\lambda=\mathbb{Z}[\lambda,\lambda^{-1}] \ltimes \langle \lambda \rangle \ltimes [0_+,1_-]$ 
 is generated by $\mathcal{S}$; $\G_\lambda$ is compactly generated.  \label{glambda compact gen}
\end{lemma}
\begin{proof}
Lemma \ref{fi} and \ref{gi} say that the (full) bisections:
$$f_i, f_{-i} \in \G_\lambda^\mathcal{S} ,$$
for all $i \in \mathbb{N}$, where $\G_\lambda^\mathcal{S}$ is the subgroupoid of $\G_\lambda$ generated by $\mathcal{S}$. It is clear that $f_i,f_{-i}$ span a subgroup of $V(\mathbb{Z}[\lambda,\lambda^{-1}],\langle \lambda \rangle , 1)$ isomorphic to $\mathbb{Z}[\lambda,\lambda^{-1}]/\mathbb{Z}$. Therefore, $\G^{\mathcal{S}}_\lambda$ has all germs of the form:
$$ ((\alpha, 1), x) \in \G^{\mathcal{S}}_\lambda.$$
Let $g=((\hat{\alpha}, \lambda^{-n}),x) \in \G_\lambda$ be arbitrary, with $n>0$. Then $0<q(r(g))=\lambda^{-n}(q(x)+\hat{\alpha})<1,$ so in particular, $0<q(x)+\hat{\alpha}<\lambda^{n}$. Then, $((\hat{\alpha},1),x) \in \G_{\lambda}^\mathcal{S}$. 
Then $q(r((\hat{\alpha},1),x)<\lambda^n \Rightarrow r((\hat{\alpha},1),x) \in s(c^{-n})$. Therefore $g=c^{-n}((\hat{\alpha},1),x) \in \G_\lambda^\mathcal{S}$.
\end{proof}

 It remains to understand the case when $\Lambda$ is finitely but not singly generated. In fact, we can reduce the case when $\Lambda$ is polycyclic to the case when $\Lambda$ is cyclic by using the following Lemma. 
 
 Let us first explain the inclusion of our groupoids in one another. Let $\G=\mathbb{Z}[\lambda_1^{\pm 1}, \lambda_n^{\pm 1}] \ltimes \langle \lambda_1,...,\lambda_n \rangle \ltimes [0_+,1_-]$ be a groupoid where $\Lambda$ is polycyclic. Let $N=\{\lambda_{k_1}, ..,\lambda_{k_K}\}$ be a finite collection of $\lambda_i$. Let $B$ be a compact open bisection in $\G_N=\mathbb{Z}[\lambda_{k_1}^{\pm 1}, \lambda_{k_K}^{\pm 1}] \ltimes \langle \lambda_{k_1},...,\lambda_{k_K} \rangle \ltimes [0_+,1_-]$. We can consider $B$ as a compact open bisection in $\G$ via the canonical inclusion. Moreover, the canonical inclusion map is an inclusion on the level of the inverse semigroup of compact open bisections, so that if $K$ was a compact generating set for $\G_N$ then $K$ is a compact set in $\G$ which generates a subgroupoid containing all compact bisections of the form $((c,\mu),[a_+,b_-])$ where $(c,\mu) \in \mathbb{Z}[\lambda_{k_1}^{\pm 1}, \lambda_{k_K}^{\pm 1}] \ltimes \langle \lambda_{k_1},...,\lambda_{k_K} \rangle$ and $a,b \in \mathbb{Z}[\lambda_1^{\pm 1}, \lambda_n^{\pm 1}]$ are arbitrary. 

\begin{lemma}
    Let $\lambda_1,...,\lambda_N \in [0,1]$ be a collection of real numbers. 
    \label{lemma where we reduce to singly generated}
    Then, the groupoid $\mathbb{Z}[\lambda_1^{\pm 1},...,\lambda_N^{\pm 1}] \ltimes \langle \lambda_1,...,\lambda_N \rangle  \ltimes [0_+,1_-]$ is compactly generated, therefore $D(V(\langle \lambda_1,...,\lambda_N \rangle ,\mathbb{Z}[\lambda_1^{\pm 1},...,\lambda_N^{\pm1}],1))$ is finitely generated.
\end{lemma}
\begin{proof}
 Let us prove this by induction. The base case $N=1$ is covered by Lemma \ref{glambda compact gen} and Corollary \ref{finitely generated compactly}. Assume true for $N$, let us show for $N+1$. 

 By inductive hypothesis there is a compact generating set $\mathcal{S}_N$ for the groupoid $\mathbb{Z}[\lambda_1^{\pm 1},...,\lambda_N^{\pm 1}] \ltimes \langle \lambda_1,...,\lambda_N \rangle  \ltimes [0_+,1_-]$
 By Lemma \ref{glambda compact gen} and Corollary \ref{finitely generated compactly} there is a compact generating set $\mathcal{S}_{N+1}$ which generates the groupoid $\mathbb{Z}[\lambda_{N+1},\lambda_{N+1}^{-1}] \ltimes \langle \lambda_{N+1} \rangle \ltimes [0_+,1_-]$. Let $\mathcal{S}=\mathcal{S}_N \sqcup \mathcal{S}_{N+1}$, considering the bisections in the subgroupoids as bisections in the enveloping groupoid.  Let $\G^{\mathcal{S}} \subset \G:=\mathbb{Z}[\lambda_1^{\pm 1},...,\lambda_{N+1}^{\pm 1}] \ltimes \langle \lambda_1,...,\lambda_{N+1} \rangle  \ltimes [0_+,1_-]$ be the subgroupoid generated by $\mathcal{S}$. By construction then, we have that \begin{equation}   
 (c,\mu) \in \mathbb{Z}[\lambda_{N+1},\lambda_{N+1}^{-1}],  \ltimes \langle \lambda_{N+1} \rangle,  a,b \in \mathcal{Z}[\lambda_1^{\pm 1} ...,\lambda_{N+1}^{\pm 1}] \Rightarrow ((c,\mu), [a_+,b_-])  \in \mathcal{G}^\mathcal{S} \label{n+1 incl}\end{equation}
 and, \begin{equation}
    (c,\mu) \in \mathbb{Z}[\lambda_1^{\pm 1},...,\lambda_N^{\pm 1}] \ltimes \langle \lambda_1,...,\lambda_N \rangle,  a,b \in \mathcal{Z}[\lambda_1^{\pm 1} ...,\lambda_{N+1}^{\pm 1}] \Rightarrow ((c,\mu), [a_+,b_-])  \in \mathcal{G}^\mathcal{S}. \label{sn incl} \end{equation} Let us proceed to prove two claims, in analogy to Lemma \ref{fi} and Lemma \ref{gi}. 
 \begin{enumerate}
     \item Our first claim is that for all (suitable) additive germs in $\G^\mathcal{S}$ we may multiply them by positive powers of $\lambda_N$, we proceed in analogy to Lemma \ref{fi}. More precisely, let us show that for all $\alpha \in [0,1) \cap \mathbb{Z}[\lambda_1^{\pm 1},...,\lambda_{N+1}^{\pm 1}]$ such that the bisection
     $f_\alpha =((\alpha,1),[0_+,(1-\alpha)_+]) \sqcup ((\alpha-1,1),[(1-\alpha)_-,1_+] \in \G^{\mathcal{S}}_\lambda$, we have that $f_{\lambda_{N+1}\alpha}=((\lambda_{N+1}\alpha,1), [0_+,(1-\lambda_{N+1}\alpha)_-] \sqcup ((\lambda_{N+1}\alpha-1,1),[(1-\lambda_{N+1}\alpha)_-,1_+] \in \mathcal{G}^\mathcal{S}_\lambda$.

     Let $c_{N+1}=((0,\lambda_{N+1}),[0_+,1_-])\in \mathcal{\G}^\mathcal{S}$ by (\ref{n+1 incl}). We have that $f_{\lambda_{N+1}\alpha}|_{[0_+,(\lambda_{N+1}(1-\alpha))_-]}=c_{N+1}f_\alpha c_{N+1}^{-1}|_{[0_+,(\lambda_{N+1}(1-\alpha))_-]} \in \G^\mathcal{S}$. Note 
     $f_{\lambda_{N+1}}=((\lambda_{N+1},1), [0_+,(1-\lambda_{N+1})_-] \sqcup ((\lambda_{N+1}-1,1),[(1-\lambda_{N+1})_-,1_+] \in \mathcal{G}^\mathcal{S}$ by (\ref{n+1 incl}). We have that:
     $$f_{\lambda_N\alpha}|_{[(\lambda_{N+1}(1-\alpha))_+,(\lambda_{N+1})_-]}=f_{\lambda_{N+1}}^{-1}c_{N+1}f_\alpha c_{N+1}^{-1}|_{[(\lambda_{N+1}(1-\alpha))_+,(\lambda_{N+1})_-]}$$
     Therefore, $f_{\lambda_{N+1}}|_{[0_+,(\lambda_{N+1})_-]}= f_{\lambda_{N+1}\alpha}|_{[0_+,(\lambda_{N+1}(1-\alpha))_-]}
 \sqcup f_{\lambda_N\alpha}|_{[(\lambda_{N+1}(1-\alpha))_+,(\lambda_{N+1})_-]} \in \G^\mathcal{S}$.

      We have that for all $n \in 0,1,...,\lfloor \lambda_{N+1}^{-1} \rfloor -1$
     $$f_{\lambda_{N+1}\alpha}|_{[(n\lambda_{N+1})_+,((n+1)\lambda_{N+1})_-]}=f_{\lambda_{N+1}}^{n} f_{\lambda_N\alpha}|_{[0_+,(\lambda_N)_-]} f_{\lambda_{N+1}}^{-n} \in \G^\mathcal{S} $$
     And for $n=\lfloor \lambda_{N+1}^{-1} \rfloor$, 
     $$f_{\lambda_N\alpha}|_{[(n\lambda_{N+1})_+,1_-]}=f_{\lambda_{N+1}}^{n} f_{\lambda_N\alpha}|_{[0_+,(\lambda_N)_-]} f_{\lambda_{N+1}}^{-n} |_{[(n\lambda_{N+1})_+,1_-]} \in \G^\mathcal{S}. $$
     \item Our second claim is that for all (suitable) additive germs in $\G^\mathcal{S}$ we may multiply them by negative powers of $\lambda_N$, proceeding in analogy to Lemma \ref{gi}. More precisely, let us show that for all $\alpha \in [0,\lambda_{N+1}) \cap \mathbb{Z}[\lambda_1^{\pm 1},...,\lambda_{N+1}^{\pm 1}]$ such that the bisection
     $f_\alpha =((\alpha,1),[0_+,(1-\alpha)_+]) \sqcup ((\alpha-1,1),[(1-\alpha)_-,1_+] \in \G^\mathcal{S}$, we have that $f_{\lambda_{N+1}^{-1}\alpha}=((\lambda_{N+1}^{-1}\alpha,1), [0_+,(1-\lambda_{N+1}^{-1}\alpha)_-] \sqcup ((\lambda_{N+1}^{-1}\alpha-1,1),[(1-\lambda_{N+1}^{-1}\alpha)_-,1_+] \in \mathcal{G}^\mathcal{S}$.

     Let $c_{N+1}=((0,\lambda_{N+1}),[0_+,1_-])\in \mathcal{\G}^\mathcal{S}$ by (\ref{n+1 incl}). We have that:
     $$ f_{\lambda_{N+1}^{-1}\alpha}|_{[0_+,(1-\lambda_{N+1}^{-1}\alpha)_-]}=c_{N+1}^{-1} f_{\alpha}c_{N+1}|_{[0_+,(1-\lambda_{N+1}^{-1}\alpha)_-]} $$
     Let $K$ be large enough that $\lambda_{N+1}^K< 1-\lambda_{N+1}^{-1} \alpha $. We have that $$f_{\lambda_{N+1}^K}=((\lambda_{N+1}^K,1), [0_+,(1-\lambda_{N+1}^K)_-] \sqcup ((\lambda_{N+1}^K-1,1),[(1-\lambda_{N+1}^K)_-,1_+] \in \mathcal{G}^\mathcal{S}$$ by (\ref{n+1 incl}). Then, conjugating $f_{\Lambda_{N+1}^{-1}\alpha}|_{[0_+,(1-\lambda_{N+1}^{-1}\alpha)_-]}$ by $(f_{\lambda_{N+1}^K})^n$, as before, will cover $f_{\lambda_{N+1}^{-1}\alpha}$ on $[0_+,1_-]$. 
 \end{enumerate}
Using the above claims we next establish that for all $m_1,....,m_{N+1} \in \mathbb{Z}$, for $\mu=\prod_{i=1}^{N+1} \lambda_i^{m_i} $ we have that the bisection:
$$ f_{\mu-\lfloor \mu \rfloor}=((\mu-\lfloor \mu \rfloor,1), [0_+,(1-\mu+\lfloor \mu \rfloor)_-] \sqcup ((\mu-\lfloor \mu \rfloor-1,1),[(1-\mu+\lfloor \mu \rfloor)_-,1_+] \in \mathcal{G}^\mathcal{S}$$
Case where $|m_{N+1}|=0$ follows by (\ref{sn incl}), since $\mu-\lfloor \mu \rfloor \in  \mathbb{Z}[\lambda_1^{\pm 1},...,\lambda_N^{\pm 1}] $. Let us prove the inductive step separately for the cases $m_{N+1}>0$ and $m_{N+1}<0$
\begin{itemize}
    \item Let us do the case when $m_{N+1}>0$ by induction. Assuming we have for $\nu=\lambda_{N+1}^{-1}\mu$, $f_{\nu-\lfloor \nu \rfloor} \in \G^\mathcal{S}$. 
    
    Then use claim 1) from above, we have that $f_{\lambda_{N+1}\mu -\lfloor \lambda_{N+1}^{-1} \mu \rfloor \lambda_{N+1}} \in \G^\mathcal{S}$. But we have already established that 
    $f_{\lambda_{N+1}} \in \G^\mathcal{S}$, hence 
    
    $f_{\mu -\lfloor \mu \rfloor }=f_{\lambda_{N+1}}^{\lfloor \lambda_{N+1}^{-1} \mu \rfloor} f_{\mu-\lfloor \lambda_{N+1}^{-1} \mu \rfloor \lambda_{N+1}} \in \G^\mathcal{S} $
    \item The case when $m_{N+1}<0$ is completely analogous, assume that we have $ f_{\lambda_{N+1}\mu-\lfloor \lambda^{N+1}\mu \rfloor} \in \G^{\mathcal{S}}$ we have that  $\lambda_{N+1}\mu-\lfloor \lambda^{N+1}\mu \rfloor \in [n \lambda_{N+1} , (n+1) \lambda_{N+1}] $, for some $n$. Hence, 
    $$f_{\lambda_{N+1}\mu-\lfloor \lambda^{N+1}\mu \rfloor-n\lambda_{N+1}}=f_{\lambda_{N+1}}^{-n}f_{\lambda_{N+1}\mu-\lfloor \lambda_{N+1}\mu \rfloor} \in \G^\mathcal{S} $$
    Now using claim 2) we have that 
    $$f_{\mu-\lambda_{N+1}^{-1} \lfloor \lambda_{N+1} \mu \rfloor -n} \in \G^\mathcal{S}$$
    But note that $f_{\lambda_{N+1}^{-1}-\lfloor \lambda_{N+1} \rfloor } \in \G^\mathcal{S}$, it is generated by elements in $\mathcal{S}_{N+1}$. Hence, we have that: 
    $$f_{\mu- \lfloor \mu \rfloor }=f_{\lambda_{N+1}^{-1}-\lfloor \lambda_{N+1} \rfloor}^{\lfloor \lambda_{N+1} \mu \rfloor}f_{\mu-\lambda_{N+1}^{-1} \lfloor \lambda_{N+1} \mu \rfloor -n}  \in \G^\mathcal{S}$$
\end{itemize}
For all $\mu$ as above, we have $f_{\mu-\lfloor \mu \rfloor } \in \G^\mathcal{S}$. These bisections span an abelian subgroup of the topological full group 
$V(\langle \lambda_1,...,\lambda_{N+1} \rangle, \mathbb{Z}[\lambda_1^{\pm 1 }, ...,\lambda_{N+1}^{\pm 1}], 1 )$ isomorphic to $\mathbb{Z}[\lambda_1^{\pm 1 }, ...,\lambda_{N+1}^{\pm 1}]/\mathbb{Z}$

Let $\mu \in \langle \lambda_1,...,\lambda_{N+1} \rangle$. Then we may rewrite $\mu$ as $\mu_N \lambda_{N+1}^k$ where $k \in \mathbb{Z}$ and $\mu_{N} \in \langle \lambda_1,...,\lambda_N \rangle $. Then, $c_{\mu_N}=((0,\mu_N),[0_+,1_-]\cap[0_+,\mu_N^{-1}]) \in \G^\mathcal{S}$ by (\ref{sn incl}) and $c_{N+1}=((0,\lambda_{N+1}),[0_+,1_-]) \in \G^\mathcal{S}$ by (\ref{n+1 incl}).
If $\lambda^k>\mu_N$, then $c_\mu=((0,\mu),[0_+,1_-] \cap[0_+,\mu^{-1}_-])=c_{N+1}^kc_{\mu_N} \in \G^\mathcal{S}$. Otherwise $c_{\mu}=c_{\mu_N}c_{N+1}^k \in \G^\mathcal{S}$.

Therefore to obtain any line segment corresponding to any $(a,\mu) \in \mathbb{Z}[\lambda_1,...,\lambda_{N+1}] \ltimes \langle \lambda_1,...,\lambda_{N+1} \rangle $, where $\mu<1$ one may write in the form $f_{\mu^{-1}a}c_\mu|_{D(a,\mu)} $ where $D(a,\mu)$ is the maximal domain. However, up to taking inverses, this is all possible slopes, hence, $\G^\mathcal{S}=\G=\mathbb{Z}[\lambda_1^{\pm 1},...,\lambda_{N+1}^{\pm 1}] \ltimes \langle \lambda_1,...,\lambda_{N+1} \rangle  \ltimes [0_+,1_-]$, completing our inductive step and our proof. 
\end{proof}

We summarise our discussion in the below theorem, by combining the above Lemma \ref{lemma where we reduce to singly generated} and Corollary \ref{finitely generated compactly}. This Theorem is Theorem \ref{thmintro1}.
\begin{theorem}
\label{fg when fg by alg}
   Let $\Lambda$ be a subgroup of $(\mathbb{R}_+, \cdot)$ and $\Gamma$ be a $\mathbb{Z} \cdot \Lambda$ submodule. The following are equivalent:
   \begin{itemize}
       \item $\Gamma \ltimes \Lambda$ is finitely generated.
       \item $\Gamma \ltimes \Lambda \ltimes [0_+,\ell_-]$ is compactly generated for all $\ell \in \Gamma$. 
       \item $D(V(\Gamma,\Lambda,\ell))$ is finitely generated for all $\ell \in \Gamma$.
        \item $D(V(\Gamma,\Lambda,\ell))$ is 2 generated for all $\ell \in \Gamma$. 
   \end{itemize}
\end{theorem}
Note that our discussion does not investigate the higher finiteness properties of Stein's groups, but these have been partially studied. Notably, Stein showed the derived subgroup is type $F_\infty$ in the case when $\Lambda$ is generated by finitely many integers, $\Gamma=\mathbb{Z} \cdot \Lambda$, and $\ell \in \mathbb{N}$ \cite{stein1992groups}. We therefore ask:
\begin{question}
Suppose $\Lambda$ is a finitely generated subgroup of $(\mathbb{R}_+, \cdot)$ let $( \mathbb{Z} \cdot \Lambda,+)$ denote the group ring. Under what conditions is $D(V(\Gamma,\Lambda,\ell))$ finitely presented?
\end{question}

\section{Homology of Stein's groups}

\subsection{Initial observations}
Recently, there have been vast developments in our understanding of the relationship between the homology of groupoids and their topological full groups \cite{li2022}. In the case of the partial actions we study, the groupoid homology of globalisable partial actions, groupoid homology reduces to group homology. 
\begin{lemma}
    Let $\alpha:G \acts Y$ be a group action on a totally disconnected space and let $X \subset Y$. Then for the partial action groupoid, we have the following:
    $$ H_*(G \ltimes_\alpha X) \cong H_*(G \ltimes_\alpha Y) \cong H_*(G, C_c(Y) ) $$
    Where $H_*(G \ltimes_\alpha X), H_*(G \ltimes_\alpha Y)$ denote the groupoid homology of $G \ltimes_\alpha X, G\ltimes_\alpha Y$, and $H_*(G,C_c(Y))$ is the classical group homology with coefficients in $C_c(Y)$.
\end{lemma}
See for example [\cite{matui2012homology}, Theorem 3.6]. Therefore the homology of the groupoids $\Gamma \ltimes \Lambda \ltimes_\beta [0_+,\ell_-]$ is related to the homology of the groups $\Gamma \ltimes \Lambda$:
    $$H_*(\Gamma \ltimes \Lambda \ltimes_\beta [0_+,\ell_-])=H_*(\Gamma_\lambda \ltimes \Lambda_\lambda , C_c(\mathbb{R}_\Gamma,\mathbb{Z}))$$
 Let us try to compute these homology groups. We begin with our analogy of Lemma 5.5. \cite{xinlambda}, wherein we compute $H_0(\Gamma \ltimes \Lambda \ltimes_\beta [0_+,\ell_-])$. 
\begin{lemma}
Let $\Lambda= \langle \lambda_1 ,...,\lambda_k \rangle$, $\Gamma= \mathbb{Z}[ \lambda_1^{\pm 1},...., \lambda_k^{\pm 1}]$ be the canonical groups and submodules generated by finitely many algebraic numbers, assuming without loss of generality that the finite collection of algebraic numbers $\{\lambda_i\}_{k=1}^N$ are pairwise algebraically independent. Then,
$$H_0(\Gamma \ltimes \Lambda, C_c(\mathbb{R}_\Gamma, \mathbb{Z})) \cong \Gamma/N_\Lambda$$
Where $N_\Lambda$ is the normal subgroup given by $N_\Lambda=\sum_i (1-\lambda_i) \Gamma$. Considering $H_0$ as an ordered group, the order unit is given by the equivalence class of $\ell$, $[\ell] \in \Gamma/N_\Lambda$.

Moreover, $H_1(\Gamma \ltimes \Lambda, C_c(\mathbb{R}_\Gamma, \mathbb{Z}))$ is finitely generated. 
\label{H0 computation}
\end{lemma}
\begin{proof}
Let $\hat{\mathbb{R}}_\Gamma=\mathbb{R}_\Gamma \cup \{\infty\}  $, and let us extend $\beta$ to an action of  $\Gamma \ltimes \Lambda \acts \hat{\mathbb{R}}_\Gamma$ by setting $\infty$ to be a fixed point. Let us topologise $\hat{\mathbb{R}}_\Gamma$ with the order topology, where we say that $\infty>a$, $a \in \mathbb{R}_\Gamma $.  
Then, a $\mathbb{Z}$-basis of $C_c(\hat{\mathbb{R}}_\Gamma,\mathbb{Z})$ which $\Gamma$ will always act on freely and transitively is given by $$\chi_{(a_+,\infty]}(t)=\begin{cases}
    1 & a_+<t \leq \infty \ \\ 0 & \text{else}
\end{cases} \quad a \in \Gamma .$$ For this reason, $C_c(\hat{\mathbb{R}}_\Gamma,\mathbb{Z}) \cong \mathbb{Z} \Gamma$. Consequently, 
$$H_q (\Gamma, C_c(\hat{\mathbb{R}}_\Gamma, \mathbb{Z}))=\begin{cases} \mathbb{Z} & q=0 \ \\
0 & \text{ else }
\end{cases}$$
Now note by construction of $\Gamma \ltimes \Lambda$ we have the split short exact sequence:
$$1 \rightarrow \Gamma \rightarrow \Gamma \ltimes \Lambda \rightarrow \mathbb{Z}^n \rightarrow 1$$
It follows by the Hochschild-Serre spectral sequence (Chapter VII, Theorem 6.3) that
$$H_i(\Gamma \ltimes \Lambda, C_c(\hat{\mathbb{R}}_\Gamma, \mathbb{Z}) ) \cong H_i (\Lambda )=H_i (\mathbb{Z}^n)=\mathbb{Z}^{^{i}C_n}$$
The final equality follows by the Kunneth formula, and this is where $n$ is the number of generators of $\Lambda$. Now let us apply the Hochschild-Serre spectral sequence again, this time choosing as the module $M=\mathbb{Z}=M_\Gamma$. Then all coefficients are in $\mathbb{Z}$. We obtain the following exact sequence:
$$ H_2(\Gamma \ltimes \Lambda) \rightarrow H_2(\mathbb{Z}^n)\rightarrow H_1 (\Gamma)_{\mathbb{Z}^n}  \rightarrow H_1(\Gamma \ltimes \Lambda)\rightarrow H_1 (\mathbb{Z}^n) \rightarrow 0   $$
But since the SES of groups with $\Gamma \ltimes \Lambda$ in the middle splits, the map from $H_2(\Gamma \ltimes \Lambda) \rightarrow H_2(\mathbb{Z}^n)$ is a surjection. Therefore, the next map must be the zero map, revealing the short exact sequence:
$$0 \rightarrow H_1(\Gamma)_{\mathbb{Z}^n}=\Gamma \rightarrow H_1(\Gamma \ltimes \Lambda) \rightarrow \mathbb{Z}^n \rightarrow 0.  $$
Therefore, $H_1(\Gamma \ltimes \Lambda)=\mathbb{Z}^n \oplus \Gamma/N$ where $N=\sum_{i}(1-\lambda_i)\Gamma$
Then,
$$0 \rightarrow C_c(\mathbb{R}_\Gamma, \mathbb{Z}) \rightarrow C_c (\hat{\mathbb{R}}_\Gamma, \mathbb{Z}) \rightarrow \mathbb{Z} \rightarrow 0 $$
Is a short exact sequence of $\mathbb{Z}(\Gamma \ltimes \Lambda)$ modules. By proposition 6.1, Chapter III of Brown, we get that there is a long exact sequence:
$$\rightarrow H_1(\Gamma \ltimes \Lambda, C_c(\hat{\mathbb{R}}_\Gamma,\mathbb{Z}))\rightarrow H_1(\Gamma \ltimes \Lambda) \rightarrow H_0(\Gamma \ltimes \Lambda,C_c(\mathbb{R}_\Gamma,\mathbb{Z})) \rightarrow  H_0(\Gamma \ltimes \Lambda, C_c(\hat{\mathbb{R}}_\Gamma,\mathbb{Z}))\rightarrow H_0(\Gamma \ltimes \Lambda) \rightarrow 0$$
Plugging in what we know: 
$$ \rightarrow \mathbb{Z}^n \rightarrow \mathbb{Z}^n \oplus \Gamma/N  \xrightarrow{f} H_0(\Gamma \ltimes \Lambda,C_c(\mathbb{R}_\Gamma,\mathbb{Z})) \xrightarrow{g}  \mathbb{Z} \rightarrow \mathbb{Z} \rightarrow 0 $$
Is an exact sequence. The map $\mathbb{Z}^n \rightarrow \mathbb{Z}^n \oplus \Gamma/N$ is just the canonical inclusion $z \mapsto (z,0)$. Then the map into $H_0( \Gamma \ltimes \Lambda,C_c(\mathbb{R}_\Gamma\mathbb{Z}))$  is $0$ exactly on $\mathbb{Z}^n \oplus 0$; the kernel is $\Gamma/N$. 
Also, the map $\mathbb{Z} \rightarrow \mathbb{Z}$ must be an isomorphism, so one has that $g$ is the zero map. But also, $im(f)=ker(g)$. Therefore, by the first isomorphism theorem, we get that $H_0(\Gamma \ltimes \Lambda,C_c(\mathbb{R}_\Gamma,\mathbb{Z}) \cong Im(f)/ker(g)=Im(f)=\Gamma/N$, as required. 

Let us show the first homology group is finite rank. Let us reconsider the long exact sequence, this time around $H_1(\Gamma \ltimes \Lambda, C_c(\mathbb{R}_\Gamma, \mathbb{Z})):$
$$...H_2( \Gamma \ltimes \Lambda, C_c(\hat{\mathbb{R}}_\Gamma,\mathbb{Z}))  \rightarrow H_2(\Gamma \ltimes \Lambda) \rightarrow H_1(\Gamma \ltimes \Lambda, C_c(\mathbb{R}_\Gamma,\mathbb{Z}) \rightarrow H_1(\Gamma \ltimes \Lambda, C_c(\hat{\mathbb{R}}_\Gamma,\mathbb{Z}) \rightarrow ... $$
$H_2(\Gamma \ltimes \Lambda)$ and $H_1(\Gamma \ltimes \Lambda, C_c(\hat{\mathbb{R}}_\Gamma,\mathbb{Z}) )\cong \mathbb{Z}^n$ are both finite rank. Therefore it follows that $H_1(\Gamma \ltimes \Lambda, C_c(\mathbb{R}_\Gamma,\mathbb{Z})$ is finite rank. 

It remains to determine the position of the unit in $H_0(\Gamma \ltimes \Lambda \ltimes [0_+,\ell_-])$. This is given by considering $[0_+,\ell_-]$ as a bisection. It follows from the above computation that the position of the unit is exactly the equivalence class of $\ell$ $[\ell] \in \Gamma/\sum_{i}(1-\lambda_i)\Gamma$.
\end{proof}
It would be interesting to compute the higher groupoid homology groups in this much generality, but the methods of Lemma 5.5. \cite{xinlambda} do not generalise directly here. Therefore, one would need a new methodology for computing the group homology $H_*(\Gamma \ltimes \Lambda,C_c(\mathbb{R}_\Gamma,\mathbb{Z}))$ in general.

The good news is that even from just the computation of $H_0$, one may already distinguish many of these groups. 
\begin{corollary}
    Suppose $\Lambda=\langle \lambda_1,..., \lambda_n \rangle, \; \hat{\Lambda}= \langle \hat{\lambda}_1,..., \hat{\lambda}_m \rangle $ be multiplicative subgroups of $\mathbb{R}_+$, generated by algebraic numbers. Let $\Gamma,\hat{\Gamma}$ be (respectively) $\mathbb{Z} \cdot \Lambda, \mathbb{Z} \hat{\Lambda}$ submodules and $\ell \in \Gamma, \hat{\ell} \in \hat{\Gamma}$. Let $N_{\Lambda}= \sum_{i=1}^n (1-\lambda_i) \Gamma$, 
 and $N_{\hat{\Lambda}}=\sum_{i=1}^n (1-\lambda_i) \hat{\Gamma}$. Suppose $V(\Gamma,\Lambda,\ell) \cong V(\hat{\Gamma}, \hat{\Lambda}, \hat{\ell})$. 
 \label{cor classification}
 Then $\Gamma/N_\Lambda \cong \hat{\Gamma}/N_{\hat{\Lambda}}$
and $\ell-\hat{\ell} \in N_\Lambda$.  \label{classification corollary}
\end{corollary}
The above Corollary recovers for example, the fact observed by Higman \cite{higman1974finitely} that $V_{k,r} \cong V_{k',r'} \Rightarrow k=k'$ and $r-r' \cong 0 \mod{k \mathbb{Z}}$. 
We remark upon some transfers of homological information from $\Gamma \ltimes \Lambda \ltimes [0_+,1_-]$ to $V(\Gamma,\Lambda,\ell)$. This is \cite[Corollary E]{li2022} specialised to our particular circumstance.

\begin{theorem}[AH Exact Sequence]
  Let $\Gamma,\Lambda,\ell$ be arbitrary. Then, there is a 5-term exact sequence:
  $$ H_2(D(V(\Gamma,\Lambda,\ell)) \rightarrow H_2(\Gamma \ltimes \Lambda, C_c(\mathbb{R}_\Gamma , \mathbb{Z})) \rightarrow H_0(\Gamma \ltimes \Lambda, C_c(\mathbb{R}_\Gamma , \mathbb{Z})) \otimes \mathbb{Z}_2 \rightarrow H_1( V(\Gamma,\Lambda,\ell))=$$ $$=V(\Gamma,\Lambda,\ell)_{ab} \rightarrow H_1(\Gamma \ltimes \Lambda, C_c(\mathbb{R}_\Gamma , \mathbb{Z})) \rightarrow 0$$ \label{ah conj}
\end{theorem}
We can use this to determine that many of the $V$-type groups are finitely generated. This  concludes the proof of Theorem \ref{thmintro2}, our second main result. 
\begin{theorem}
    Let $\Lambda$ be a finitely generated multiplicative subgroup of the positive algebraic numbers. Let $\Gamma$ be a $\mathbb{Z} \cdot \Lambda$ submodule, and let $\ell \in \Gamma$. Then $V(\Gamma,\Lambda,\ell)$ is finitely generated. \label{fg v type}
\end{theorem}
\begin{proof}
 By Lemma \ref{H0 computation}, we have that whenever $\Gamma$ and $\Lambda$ are generated by finitely many algebraic numbers $H_0$ and $H_1$ of the underlying groupoids are finite rank. Plugging this into the 5-term exact sequence in Theorem \ref{ah conj},  it follows that the abelianisation is finite rank. Therefore $V(\Gamma,\Lambda,\ell)$, the extension of $D(V(\Gamma,\Lambda,\ell))$ (finitely generated by Theorem \ref{fg when fg by alg}) by $V(\Gamma,\Lambda,\ell)_{ab}$ must be finitely generated. 
\end{proof}
Using \cite[Theorem F]{li2022}, we can obtain a homological stability result that generalizes \cite[Theorem 3.6]{szymik2019homology} to this more general setting. Here we use the Thompson-like groups on noncompact intervals, since this is the most general form we can give. This appears as Corollary \ref{corhomintro} in the introduction.  
\begin{corollary}
  Let $\Gamma,\Lambda$ be arbitrary. Let $U_1,U_2 \subset \mathbb{R}$ be closed subsets with nonempty interior $\Gamma \cup \{-\infty,+\infty\}$. Then for all $*$, 
  $H_*(V(\Gamma,\Lambda,U_1)) \cong H_*(V(\Gamma,\Lambda,U_2))$ and  $H_*(D(V(\Gamma,\Lambda,U_1))) \cong H_*(D(V(\Gamma,\Lambda,U_2)))$ \label{cor hom best}
\end{corollary}
\begin{proof}
Let $U_1$ be arbitrary let $\hat{U}_1 \subset \mathbb{R}_\Gamma$ be the unique cylinder set such that $q(\hat{U}_1)=U_1$. Then the canonical inclusion:
$$\Gamma \ltimes \Lambda  \ltimes \hat{U}_1 \hookrightarrow \Gamma \ltimes \Lambda  \ltimes \mathbb{R}_\Gamma   $$
is a Morita equivalence of \'etale groupoids, inducing an isomorphism in groupoid homology. Hence by [\cite{li2022}, Theorem F] this induces an isomorphism in group homology: 
$$ H_*( V(\Gamma,\Lambda,U_1) ) \cong  H_*( V(\Gamma,\Lambda,\mathbb{R}))$$
$$ H_*( D(V(\Gamma,\Lambda,U_1)) ) \cong  H_*( D(V(\Gamma,\Lambda,\mathbb{R}))) $$
But since $U_1$ was arbitrary, we are done. 
\end{proof}
We see how to apply this result in practice, by computing the homology of the analogy of Thompson's group $V$ that acts on $\mathbb{R}$ rather than $[0,1]$. 
\begin{example}
 $V(\mathbb{Z}[1/2],\langle 2 \rangle, \mathbb{R})$ is integrally acyclic and simple.     
\end{example} 
Finally, let us apply \cite[Corollary C]{li2022}, to obtain an expression for the rational homology of Stein's groups. 
\begin{theorem}[Rational Homology Computation]
    \label{rational homology}
Let $\Gamma,\lambda,\ell$ be arbitrary. For a group $G$, let $$H_{*}^{even}(G)=\begin{cases}
    H_*(G) & *  \text{ even } \ \\
    \{0\} & \text{ else }
\end{cases} \quad H_*^{odd}(G)=\begin{cases}
    H_*(G) & * \text{ odd } \ \\
    \{0\} & \text{ else }
\end{cases} $$ 
and, let 
$$H_{*>1}^{odd}(G)=\begin{cases}
    H_*(G) & *>1 \text{ odd } \ \\
    \{0\} & \text{ else }
\end{cases} $$ 
Then,
$$H_*(V(\Gamma,\Lambda,\ell), \mathbb{Q})\cong Ext(H_{*}^{odd}(\Gamma \ltimes \Lambda , C_c(\mathbb{Z}_\Gamma, \mathbb{Q})) \otimes Sym(H_{*}^{even}(\Gamma \ltimes \Lambda , C_c(\mathbb{Z}_\Gamma, \mathbb{Q}))$$
and 
$$H_*(D(V(\Gamma,\Lambda,\ell)), \mathbb{Q})\cong Ext(H_{*>1}^{odd}(\Gamma \ltimes \Lambda , C_c(\mathbb{Z}_\Gamma, \mathbb{Q})) \otimes Sym(H_{*}^{even}(\Gamma \ltimes \Lambda , C_c(\mathbb{Z}_\Gamma, \mathbb{Q}))$$
Where $Ext,Sym$ denote respectively the Exterior and Symmetric algebras in the sense of Multilinear Algebra \cite{greub1978multilinear}. 

\end{theorem}

\subsection{Homology for Irrational Slope Thompson's groups}
Let $\lambda \in (0,1)$. Consider $\Lambda_\lambda=\langle \lambda \rangle, \Gamma_\lambda=\mathbb{Z}[\lambda,1/\lambda]$. Let $\ell \in \Gamma_\lambda$ Consider the group $V_{\lambda,\ell}:=V(\Lambda_\lambda, \Gamma_\lambda, \ell)$. Such groups are natural generalisations of Thompson's group $V$ or Cleary's group $V_\tau$, and the analogous $F$-type subgroups, also known as Bieri-Strebel groups have been studied for example in \cite{irrationalslope22}, \cite{Winstone}, \cite{BieriStrebel}, \cite{Clearymore}. Their usual name is \say{irrational slope} Thompson's groups, which can be misleading: sometimes all of their slopes are rational, e.g. $V_{2/3}$, or Thompson's group $V$ both fit into the framework of so-called irrational slope Thompson's groups. The associated groupoids and their associated C*-algebras have also been studied extensively by Li in \cite{xinlambda}. In particular, there were many concrete homology computations by Li, which we recall now.

\begin{lemma}[\cite{xinlambda},  Prop 5.5 ]

Let $\Gamma_\lambda, \Lambda_\lambda, r$ be as above. Then,
$$H_0(\Gamma_\lambda \ltimes \Lambda_\lambda \ltimes_\beta [0_+,\ell_-]) \cong  \Gamma_\lambda/(1-\lambda)\Gamma_\lambda$$
with distinguished order unit $[\ell]$. 
$$H_k(\Gamma_\lambda \ltimes \Lambda_\lambda \ltimes_\beta [0_+,\ell_-]) \cong H_{k+1}(\Gamma_\lambda \ltimes \Lambda_\lambda) \quad \forall k \geq 1$$ \label{single lambda computation}
\end{lemma}
In particular, if $\lambda$ is transcendental, we have that:
$$ H_k(\Gamma_\lambda \ltimes \Lambda_\lambda \ltimes_\beta [0_+,1_-])=\bigoplus_{i=1}^\infty \mathbb{Z} \quad \forall k $$
Let us be more explicit. If $\lambda$ be an algebraic number that is the root of the minimal polynomial $f(t)=t^d+a_{d-1}t^{d-1}+... + a_1 t +a_0$. Let $\Gamma=\mathbb{Z}[\lambda,\lambda^{-1}]$, $\Lambda=\langle \lambda \rangle$ and let $\ell$ be arbitrary. If $d<4$, homology has been computed explicitly see the table of \cite{xinlambda}, Page 19. Using these results, we may compute many of the abelianisations of irrational slope Thompson groups concretely, generalising results of \cite{higman1974finitely}, \cite{irrationalslope22}, \cite{cleary2000regular}. 
\begin{corollary}
    Let $\lambda$ have the minimal polynomial $t^d+a_{d-1}t^{d-1}+....+a_0=0$. Let $\Lambda_\lambda=\langle \lambda \rangle, \Gamma_\lambda=\mathbb{Z}[\lambda,\lambda^{-1}]$ and $\ell \in \Gamma_\lambda$.  Let $V(\Gamma_\lambda,\Lambda_\lambda,\ell)$ be the irrational slope Thompson group with slopes $\langle \lambda \rangle$. 
    \begin{itemize}
        \item If $d=1$, then $V(\Gamma_\lambda,\Lambda_\lambda, \ell)_{ab}=\begin{cases}
            \mathbb{Z}_2 & f(1) \text{ even } \ \\
            0 & f(1) \text{ odd }
        \end{cases}$
        \item If $d=2, a_0=1$, and $a_1$ is odd (i.e. the minimal polynomial is of the form $t^2+(1-2n)t+1$ where $n \in \mathbb{Z}$) then $V(\Gamma_\lambda,\Lambda_\lambda, \ell)_{ab}=\mathbb{Z}$
        \item If $d=2,  a_0 \neq 1$, then 
        $V(\Gamma_\lambda,\Lambda_\lambda,\ell)_{ab}=\begin{cases}
            \mathbb{Z}_2 \oplus \mathbb{Z}/(1+a_0) \mathbb{Z} & f(1) \text{ even } \ \\
            \mathbb{Z}/(1+a_0) \mathbb{Z} & f(1) \text{ odd }
        \end{cases}$
        \item If $d=3, a_0=-1$ and $a_2+a_1$ is odd, (i.e. the minimal polynomial is of the form $t^3+(m)t^2+(m+2n+1)t-1$) then $V(\Gamma_\lambda,\Lambda_\lambda,\ell)_{ab}=\mathbb{Z}/(a_1+a_2)\mathbb{Z}$
    \end{itemize}
    \label{abelianisation single lambda}
\end{corollary}
\begin{proof}
 These results follow from the computations as in the table of \cite{xinlambda}, Page 19, and the AH long exact sequence (Theorem \ref{ah conj}).
\end{proof}
\begin{rmk}
The above Corollary recovers the computation of the abelianisation of Clearys group seen in \cite{irrationalslope22}, and of the Higman-Thompson groups seen in \cite{higman1974finitely}. 
\end{rmk}
We see that in contrast to the Higman-Thompson groups, or Brin-Higman-Thompson groups that are either perfect or have an index 2 derived subgroup, the irrational slope Thompson groups can have a variety of abelianisations. 
In particular, not all irrational slope Thompson groups are virtually simple. Whenever $k$ is even, the Brin-Higman-Thompson group $nV_{k,r}$ is perfect. There are no examples of "true" irrational slope Thompson groups that are perfect. This is because if $d>1$, then $H_1(\Gamma_\lambda \ltimes \Lambda_\lambda \ltimes \mathbb{R}_\Gamma) \neq 0$, and so the abelianisation admits a nontrivial quotient.  

This concrete picture of groupoid homology is that allows for a rational homology computation of Stein's groups. 
\begin{corollary}[Rational Homology Computation Irrational Slope Thompson's Groups] \label{rational hom comp}
 Suppose $\lambda$ is transcendental. Then, $$H_*(V(\Gamma_\lambda,\Lambda_\lambda,\ell),\mathbb{Q})=\bigoplus_{n \in \mathbb{N}} \mathbb{Q}$$
 Suppose $\lambda$ is algebraic. Then, let $$H_{*}^{even}(\Gamma_\lambda \ltimes \Lambda_\lambda, \mathbb{Q})=\begin{cases} H_*(\Gamma_\lambda \ltimes \Lambda_\lambda, \mathbb{Q}) & * \text{even} \ \\
 0 & \text{ otherwise }\end{cases},$$ $$ H_{*>2}^{even}(\Gamma_\lambda \ltimes \Lambda_\lambda, \mathbb{Q})=\begin{cases} H_*(\Gamma_\lambda \ltimes \Lambda_\lambda, \mathbb{Q}) & *>2, \; \text{even} \ \\
 0 & \text{ otherwise }\end{cases}$$
 and 
 $$\hat{H}_*^{odd}(\Gamma_\lambda \ltimes \Lambda_\lambda, \mathbb{Q})=\begin{cases} 
 \Gamma_\lambda/(1-\lambda)\Gamma_\lambda \otimes \mathbb{Q} & *=1\ \\ 
 H_*(\Gamma_\lambda \ltimes \Lambda_\lambda , \mathbb{Q}) & *>1, \; \text{odd }\ \\
 0 & \text{otherwise} \end{cases}  $$
 then, 
 $$H_*(V(\Gamma_\lambda,\Lambda_\lambda,\ell),\mathbb{Q})=Ext(H_{*+1}^{even}(\Gamma_\lambda \ltimes \Lambda_\lambda)) \otimes Sym (\hat{H}_{*+1}^{odd}(\Gamma_\lambda \ltimes \Lambda_\lambda)) $$
 $$H_*(D(V(\Gamma_\lambda,\Lambda_\lambda,\ell),\mathbb{Q}))=Ext(H_{*+1>2}^{even}(\Gamma_\lambda \ltimes \Lambda_\lambda) )\otimes Sym (\hat{H}_{*+1}^{odd}(\Gamma_\lambda \ltimes \Lambda_\lambda))$$
 Where $Ext,Sym$ denote respectively the Exterior and Symmetric algebras in the sense of Multilinear Algebra \cite{greub1978multilinear}. 
\end{corollary}
Like the Higman-Thompson groups $V_{k,r}$, certain low degree irrational slope Thompson groups are rationally acyclic, such as taking $\lambda$ roots of polynomials of the form $t^2+a_1 t +a_0$ where $a_0 \neq 1$ (also, if $\lambda$ is the root of a polynomial of the form $t^3+a_2t^2+a_1t+a_0$ with $a_0 \neq -1$). In particular, Cleary's group is rationally acyclic, being associated to the minimal polynomial $t^2+t-1$. 
\begin{example}
    Cleary's group $V_\tau$ is rationally acyclic.
\end{example}

However, in contrast to behaviour that is seen for the rational homology of the Higman-Thompson groups \cite{szymik2019homology} or more generally the Brin-Higman-Thompson groups \cite{li2022}, irrational slope Thompson groups are not rationally acyclic in general. Concrete calculations are shown below:
    \begin{itemize}
        \item If $\lambda$ is transcendental, $$H_*(V(\Gamma_\lambda,\Lambda_\lambda,\ell), \mathbb{Q})=\begin{cases} \bigoplus_{i=1}^\infty \mathbb{Q} & * > 0  \ \\
        \mathbb{Q} & *=0 \end{cases}, \; H_*(D(V(\Gamma_\lambda,\Lambda_\lambda,\ell)), \mathbb{Q})\cong \begin{cases} \bigoplus_{i=1}^\infty \mathbb{Q} & * > 1 \ \\ 0 & *=1 \ \\
        \mathbb{Q} & *=0 \end{cases}$$
        \item If $\lambda$ has minimal polynomial of the form $t^2+a_1t+1$ we have that:
        $$H_*(V(\Gamma_\lambda,\Lambda_\lambda,\ell), \mathbb{Q})= \mathbb{Q}, \; H_*(D(V(\Gamma_\lambda,\Lambda_\lambda,\ell)), \mathbb{Q})\cong \begin{cases}  \mathbb{Q} & * \neq 1 \ \\ 0 & * = 1 \end{cases}$$
        \item If $\lambda$ has minimal polynomial of the form $t^3+a_2t+a_1t-1$
        $$H_*(V(\Gamma_\lambda,\Lambda_\lambda,\ell), \mathbb{Q})= H_*(D(V(\Gamma_\lambda,\Lambda_\lambda,\ell)), \mathbb{Q}) =\begin{cases}  \mathbb{Q} & * \neq 1 \ \\ 0 & *=1  \end{cases}$$
    \end{itemize}

Now that we have seen that we can compute many homological or topological invariants that we associate to the groupoids $\Gamma_\lambda \ltimes \Lambda_\lambda \ltimes_\beta [0_+,\ell_-]$, let us describe what we know about distinguishing groups up to isomorphism for this class. The groupoid homology recovers the initial choice of algebraic number $\lambda$. Thus, we are ready to prove Corollary \ref{corintroclass}. 
\begin{corollary}
    Let $\lambda,\mu<1$ be algebraic numbers with degree $\leq 2$ and let $\ell_1 \in \Gamma_\lambda, \ell_2 \in \Gamma_\mu$. If
     $$ V(\Gamma_\lambda,\Lambda_\lambda,\ell_1) \cong V(\Gamma_\mu,\Lambda_\mu,\ell_2)$$
    Then $\lambda=\mu$, and $\ell_1-\ell_2 \in (1-\lambda)\Gamma_\lambda$. \label{classification for low degree}
\end{corollary}
\begin{proof}
Our first step is to recover $\lambda$ from the abstract group $V(\Gamma_\lambda,\Lambda_\lambda,\ell_1)$. Let us examine the groupoid $\Gamma_\lambda \ltimes \Lambda_\lambda \ltimes [0_+,(\ell_1)_-]$, and attempt to recover the minimal polynomial from the groupoid homology. $H_*(\Gamma_\lambda \ltimes \Lambda_\lambda \ltimes [0_+,(\ell_1)_-])$. Let us consider two cases:
\begin{enumerate}
    \item If $\Gamma_\lambda \ltimes \Lambda_\lambda \ltimes [0_+,(\ell_1)_-]$ is rationally acyclic. Then:
    \begin{itemize}
        \item If $H_1(\Gamma_\lambda \ltimes \Lambda_\lambda \ltimes [0_+,(\ell_1)_-])=0 $, $d=1$ and $a_0$ is determined by $H_0(\Gamma_\lambda \ltimes \Lambda_\lambda \ltimes [0_+,(\ell_1)_-]) \cong \mathbb{Z}/(1+a_0)\mathbb{Z}$.  
        
         \item Otherwise, $d=2$ and $a_0 \neq -1$. Then $a_0$ is determined by $H_1(\Gamma_\lambda \ltimes \Lambda_\lambda \ltimes [0_+,(\ell_1)_-]) \cong \mathbb{Z}/(1-a_0) \mathbb{Z}$ and $a_1$ is determined by $H_0(\Gamma_\lambda \ltimes \Lambda_\lambda \ltimes [0_+,(\ell_1)_-]) \cong \mathbb{Z}/(1+a_1+a_0) \mathbb{Z}$
    \end{itemize}
    \item If $\Gamma_\lambda \ltimes \Lambda_\lambda \ltimes [0_+,(\ell_1)_-]$ is not rationally acyclic then we have that $H_1(\Gamma_\lambda \ltimes \Lambda_\lambda \ltimes [0_+,(\ell_1)_-])=H_2(\Gamma_\lambda \ltimes \Lambda_\lambda \ltimes [0_+,(\ell_1)_-])=\mathbb{Z}$, then $d=2$, $a_0=1$. Moreover, $H_0(\Gamma_\lambda \ltimes \Lambda_\lambda \ltimes [0_+,(\ell_1)_-])\cong \mathbb{Z}/(2+a_1)\mathbb{Z}$, which uniquely determines $a_1$.
\end{enumerate}
This shows we may recover $\lambda$ from the groupoid homology. Groupoid homology is an invariant for groupoids. By Theorem \ref{matui isomorphism theorem}, these groupoids are uniquely determined by the abstract groups $V(\Gamma,\Lambda,\ell)$ and so for any isomorphism of the abstract groups, we must necessarily have that $\lambda=\mu$. It remains to understand how the length of the underlying interval affects the isomorphism class of these groups. One may be precise here, considering $H_0(\Gamma_\lambda \ltimes \Lambda_\lambda \ltimes [0_+,\ell_-]) \cong \Gamma/(1-\lambda)\Gamma $ as an ordered group, with the class of $[0_+,\ell_-]$, $[\ell] \in \Gamma/(1-\lambda)\Gamma_\lambda$ determining the positive cone. We therefore have an isomorphism as ordered groups, this occurs if and only if $\ell-\ell' \in (1-\lambda)\Gamma_\lambda$. 
\end{proof}

This Corollary showcases the diversity of the irrational slope Thompson's group, even for the case when $\lambda$ has degree 2. We can also distinguish some irrational slope Thompson groups of higher degrees using groupoid homology. However, in general, this classification remains an interesting open question. 
\subsection{Homology for Stein's Integral Groups}

Let $N=\{n_1,...n_k\}$ be a finite collection of integers. Let $r \in \mathbb{N}$, let $\Lambda_N=\langle n_1,..,n_k \rangle$ and $\Gamma_N=\mathbb{Z}[1/(n_1 n_2...n_k)]$. Consider the group $V(\Lambda_N,\Gamma_N,r)$. These groups have been studied in detail by Stein \cite{stein1992groups}, and are known to fit into many other frameworks of generalized Thompson's groups, for example, those explored in \cite{martinez2016cohomological}.
Because they have been more heavily studied, much more is known about the finiteness properties and presentations of these groups, for example, it was shown already by Stein in \cite{stein1992groups} that they are all of type $F_\infty$. 
However, the homology of these groups remains mysterious. For us to better understand this, it is useful to change perspectives on the underlying groupoids. 

The approach we take here is related to work involving the topological full groups of $k$-graphs, which has notably also been explored (though not for this reason) in work by Lawson-Sims-Vdovina \cite{lawson2024higher}, Lawson-Vdovina \cite{lawson2020higher}, and Dilian Yang \cite{yang2022higman}. The observation that we can rephrase the groupoid model of these exact groups in the language of k-graphs is also noted in Conchita Martınez-Pérez, Brita Nucinkis and Alina Vdovina, we include it here for completeness in the literature and for our homology computations. Let $N=\{n_1,...,n_k\}$ be a collection of integers where $k>1, n_i>1, \forall k$. Consider the single vertex $k$-graph that has $n_i$ loops for each $i=1,...,k$. Let us label these loops $a^{(n_i)}_0,...,a^{(n_i)}_{n_i-1}$ and let them be colored in the $i$-th color. We then include commutation relations as follows. Between color $i$ and color $j$ there are $n_i n_j$ commutation relations, given by: 

$$a_{k}^{(n_i)} a_{l}^{(n_j)}=a_{k'}^{(j)}a_{l'}^{(i)} \iff k/n_i + l/n_i n_j=k'/n_i n_j + l'/n_i$$

For all $0<k<n_i, 0<l<n_j$. These commutation relations induce a description of the factorisation map $d$. Notice there is a relationship already between the edges $a_k^{(n_i)}$ of our k-graph and $(k,1/n_i) \in \mathbb{Z}[\prod_{i=1}^k 1/n_i]\ltimes \langle n_1,..,n_k \rangle$: these commutation relations are the same as saying, for $(k,1/n_i),(k',1/n_i), (l,1/n_j), (l',1/n_j) \in \mathbb{Z}[\prod_{i=1}^k 1/n_i]\ltimes \langle n_1,..,n_k \rangle$:
$$a_{k}^{(n_i)} a_{l}^{(n_j)}=a_{k'}^{(j)}a_{l'}^{(i)} \iff (k,1/n_i)(l,1/n_j)=(l',1/n_j)(k',1/n_i)$$

Since it is single-vertex, the paths in this $k$-graph form a left-cancellative monoid with respect to concatenation, and the identifications from our equivalence relation. This monoid considered as a left-cancellative category fits directly into Kumjian-Pask's framework of $k$-graphs \cite{kumjian2000higher}. Let us denote the above $k$-graph by $(\Lambda(n_1,..,n_k),d)$  For more information on $k$-graphs and their groupoids, we refer the reader to this source.  \ \\
Let us describe the associated path groupoid, introduced in \cite{kumjian2000higher}. For $k>1$, Let $\Omega_k$ be the small category with objects in $\mathbb{N}^k$ and morphisms $\Omega:=\{(m,n) \; : \mathbb{N}^k \times \mathbb{N}^k, m \leq n \}$, the range and source maps of this morphism being $r(m,n)=n, s(m,n)=n$. Let $d(m,n)=n-m$. 

The path space $\Lambda^\infty(n_1,..,n_k)$ refers to all $k$-graph morphisms (that is, structure-preserving morphisms):
$$x: (\Omega_k,d) \rightarrow (\Lambda(n_1,...,n_k), d) $$
The topology on this is the natural cylinder sets (paths that begin with a finite path $\lambda$):
$$ Z(\lambda)=\{ \lambda x \in \Lambda^\infty(n_1,..,n_k) \} $$

For each $\boldsymbol{m} \in \mathbb{N}^k$ we associate a shift map
$ \sigma^{\boldsymbol{m}}: \Lambda^\infty(n_1,..,n_k) \rightarrow \Lambda^\infty(n_1,...,n_k) $ $x \mapsto \sigma^{\boldsymbol{m}}x$, where $ \sigma^{\boldsymbol{m}}x$ is the function given by

$$ \sigma^{\boldsymbol{m}}x(\boldsymbol{n_1},\boldsymbol{n_2})=x(\boldsymbol{n_1+m},\boldsymbol{n_2+m})$$
Notice then that this is a semigroup homomorphism:  $\sigma^{\boldsymbol{m}_1 + \boldsymbol{m}_2}=\sigma^{\boldsymbol{m}_1}\circ \sigma^{\boldsymbol{m}_2}$. 
Our groupoid is of the form:
$$\G(n_1,...,n_k):=\{(x,\boldsymbol{m_1-m_2},y) \; x,y \in \Lambda^\infty(n_1,...,n_k) \; \boldsymbol{m_1,m_2} \in \mathbb{N}^k, \; \sigma^{\boldsymbol{m_1}}(x)=\sigma^{\boldsymbol{m_2}}(y)  \} $$

The source of $(x,\boldsymbol{m_1-m_2},y) \in \G(n_1,...,n_k)$ is $(x,\boldsymbol{0}-\boldsymbol{0},x)$. The range of $(x,\boldsymbol{m_1-m_2},y) \in \G(n_1,...,n_k)$ is $(y,\boldsymbol{0}-\boldsymbol{0},y)$. The inverse of $(x,\boldsymbol{m_1-m_2},y) \in \G(n_1,...,n_k)$ is $(y,\boldsymbol{(-m_2)-(-m_1)},x) \in \G(n_1,...,n_k)$. 
Our unit space is identified with the path space by the canonical identification, as with Deacounu-Renault groupoids of graphs ( $(x,0,x) \mapsto x, \forall x \in \Lambda^\infty$).

A basis for the topology is coming from the open bisections (the groupoid is \'etale), which are the inverse semigroup generated by partial bijections of the form:
$$ \mathbb{Z}(\lambda,\mu)=\{ (x,\boldsymbol{m_1-m_2},y) \; x \in Z(\lambda) \; y \in Z(\mu), \; \boldsymbol{m_1,m_2} \in \mathbb{N}^k, \; \sigma^{\boldsymbol{m_1}}(x)=\sigma^{\boldsymbol{m_2}}(y) \}$$

Now let us describe why this groupoid agrees with our description of the same groupoid, namely $\mathbb{Z}[\prod_{i=1}^k 1/n_i] \rtimes \langle n_1,...,n_k \rangle \rtimes [0_+,1_-]$. The first step is to identify the unit spaces. To do this, let us notice the following distinguished elements of $\Lambda^\infty$:
$$ x_0: \mathbb{N}^k \mapsto   \Lambda(n_1,..,n_k)  \quad  (l_1,l_2..,l_k) \mapsto (a^{(n_1)}_0)^{l_1} (a^{(n_2)}_0)^{l_2}...(a^{(n_k)}_0)^{l_k} $$
$$ x_1: \mathbb{N}^k  \mapsto \Lambda(n_1,..,n_k) \quad (l_1,l_2..,l_k) \mapsto (a^{(n_1)}_{n_1-1})^{l_1} (a^{(n_2)}_{n_2-1})^{l_2}...(a^{(n_k)}_{n_k-1})^{l_k}$$
Note that since for all $i,j$ $a_0^{(n_i)}a_{0}^{(n_j)}=a_0^{(n_j)}a_{0}^{(n_i)}$ and $a_{n_i-1}^{(n_i)}a_{n_j-1}^{(n_j)}=a_{n_j-1}^{(n_j)}a_{n_i-1}^{(n_i)}$ so that these are well defined paths. $x_0$ will map to $0_+$ and $x_1$ to $1_-$. Let us first define a map on finite paths $$\phi_0: \Lambda(n_1,...,n_k) \rightarrow \mathbb{Z}[\prod_{i=1}^k n_i]  \quad (a_{l_1}^{(n_{i_1})},...,a_{l_m}^{(n_{i_m}})) \mapsto \sum_{i=1}^m n_i \prod_{j=1}^{i} \frac{1}{n_j}$$
Due to the commutation relations, this map is well-defined and bijective. 
From here, the map is extended to become $\varphi_0$. 
$$\varphi_0:\Lambda^\infty(n_1,...,n_k)  \rightarrow [0_+,1_-] \subset \mathbb{R}_{\mathbb{Z}[\prod_{i=1}^k 1/n_i]}$$ $$  x=[(a_{l_1}^{(n_{i_1})},a_{l_2}^{(n_{i_2})},...)] \mapsto \begin{cases}
   \phi_{0}(a_{l_1}^{(n_{i_1})},..,a_{l_m}^{(n_{i_m})}  )_+  & x = (a_{l_1}^{(n_{i_1})},..,a_{l_m}^{(n_{i_m})} ) x_0 \ \\ 
   (\prod_{i=1}^m\frac{1}{n_i}+\phi_0(a_{l_1}^{(n_{i_1})},..,a_{l_m}^{(n_{i_m})} )_-  & x= (a_{l_1}^{(n_{i_1})},..,a_{l_m}^{(n_{i_m})} ) x_1 \ \\
   \sum_{m=1}^\infty \phi_0 (a_{l_1}^{(n_{i_1})},..,a_{l_m}^{(n_{i_m} ) }) & otherwise 
\end{cases} $$
Remark that this does not depend on the choice of representative for $x$, due to our canonical commutation relations. Note also this map is a homeomorphism, since it is a continuous bijection between compact Hausdorff spaces: $$\varphi_0(Z(a_{l_1}^{(n_{i_1})},..,a_{l_m}^{(n_{i_m})} ))=[\phi_{0}(a_{l_1}^{(n_{i_1})},..,a_{l_m}^{(n_{i_m})}  )_+ , (\prod_{i=1}^m\frac{1}{n_i}+\phi_0(a_{l_1}^{(n_{i_1})},..,a_{l_m}^{(n_{i_m})} )_-] $$
Every open set in $[0_+,1_-]$ can be written as a finite union of sets of the 
 above form since $\phi_0$ is bijective. The next step is to identify the arrow spaces. Consider the map:
 $$ \mu: \mathbb{Z}^k \rightarrow \langle n_1,...,n_k \rangle \quad (m_1,..,m_k) \mapsto \prod_{i=1}^k n_i^{m_i}$$

$$\varphi: \G(n_1,...,n_k) \mapsto \mathbb{Z}[\prod_{i=1}^k 1/n_i] \rtimes \langle n_1,...,n_k \rangle \rtimes [0_+,1_-]$$  $$(x, \boldsymbol{m_1-m_2},y) \mapsto ( (-\varphi_0(x)+\frac{\varphi_0(y)}{\mu(\boldsymbol{m_1-m_2})},\mu(\boldsymbol{m_1-m_2})), \varphi_0(x))$$

Notice then in particular: 
$$\varphi(Z(\emptyset,a_0^{n_i}))=((0,1/n_i),[0_+,1_-])$$
$$ \varphi(Z(a_{n_i-1}^{(n_i)},a_{n_i-1}^{(n_i)}) \bigsqcup\sqcup_{l=0}^{n_i-2}Z(a_{l}^{(n_i)},a_{l+1}^{(n_i)}))$$ $$ =((1/n_i-1,1),[(1-1/n_i)_+,1_-] )\sqcup((1/n_i,1),[0_+,(1-1/n_i)_-])$$

As shown in the previous section (Theorem \ref{fg when fg by alg}), these bisections form a compact generating set of the groupoid $\mathbb{Z}[\prod_{i=1}^k 1/n_i] \rtimes \langle n_1,...,n_k \rangle \rtimes [0_+,1_-]$, hence $\varphi$ is a groupoid conjugacy. The discussion above can be summarised in the form of the Lemma below. 
\begin{lemma}
    Let $k>1$, and $N=\{n_1,...,n_k\}$ be a collection of algebraically independent integers. Then Stein's integral group $V(\Gamma_N,\Lambda_N,1)$ the topological full group of the path groupoid model of the $k$-graph above, $\G(n_1,...,n_k)$. 
    \label{identification with k-graph}
\end{lemma}
This identification will give us many interesting new facts about Stein's groups.
Notably, this identification is actually an identification of the underlying left cancellative monoids- the paths space of the $k$-graph may be identified with the underlying monoid seen in the universal groupoid description.

\begin{lemma}[Homology of groupoids with integral slope sets]
Let $N=\{n_1,...,n_k\}$ be a finite collection of integers, where $k>1$. Let $d=gcd(n_1-1,...,n_k-1)$. Let $\ell \in \Gamma_N$ be arbitrary. 
$$ H_*(\Gamma_N \ltimes \Lambda_N \ltimes [0_+,\ell_-]) \cong \begin{cases} (\mathbb{Z}/d\mathbb{Z})^{^kC_*} & *=0,...,k-1 \ \\ 0 & *\geq k \end{cases}  $$
\label{homology for stein groupoid}
\end{lemma}
\begin{proof}
    First let us apply Lemma \ref{cor hom best}. We have that: 
    $$ H_*(\Gamma_N \ltimes \Lambda_N \ltimes [0_+,\ell_-]) \cong H_*(\Gamma_N \ltimes \Lambda_N\ltimes [0_+,1_-])   $$
    Then by combining Lemma \ref{identification with k-graph} and Corollary \ref{unique groupoid conjugacy}, we have that:
    $$  H_*(\Gamma_N \ltimes \Lambda_N \ltimes [0_+,\ell_-]) \cong H_*(\G(  n_1,n_2,...,n_k))   $$
    The homology of the groupoids of single vertex k-graphs was computed in \cite{kgraphcomp}, and was shown to not depend on the underlying commutation relations. Therefore, let $d=gcd(n_1-1, ..., n_k-1)$. We have that: 
$$ H_*(\G(n_1,..,n_k)):=\begin{cases} (\mathbb{Z}/d\mathbb{Z})^{^{k-1} C_i} & i=0,...,k-1 \ \\ 0 & i \geq k \end{cases}$$
\end{proof}
This allows us to prove Corollary \ref{cor intro acyclic}
\begin{corollary}
 Let $N=\{n_1,..,n_k\}$ be a finite collection of integers with $k>1$. Let $\ell \in \Gamma_N$. Then $V(\Gamma_N ,\Lambda_N ,\ell)$ is rationally acyclic. Moreover, $V(\Gamma_N, \Lambda_N,\ell)$ is integrally acyclic if and only if $d=gcd(n_1-1, ..., n_k-1)=1$. 
 \label{acyclicty for stein}
\end{corollary}
\begin{proof}
 Observe that in Lemma \ref{homology for stein groupoid}, we have that the groupoid  
 $\Gamma_N \ltimes \Lambda_N\ltimes [0_+,\ell_-]$
is always rationally acyclic, and is acyclic if and only if $d=gcd(n_1-1, ..., n_k-1)=1$. Then, \cite{li2022} Corollary C confirms that the topological full group is always rationally acyclic. Similarly, Corollary D confirms that the topological full group is acyclic if $d=1$. Finally, if $d \neq 1$, the original abelianization computation of Stein \cite{stein1992groups}, shows that in this case, the group is not acyclic. 
\end{proof}

We end this section by remarking that this realisation of Stein's groups gives a different perspective on two of Stein's key results in \cite{stein1992groups}, namely the higher finiteness properties of Stein's integral groups and the computation of their abelianisation. 
\begin{rmk}[Alternative Proof of Abelianisation]

  Let $N=\{n_1,..,n_k\}$ with $k>1$. We can recover Stein's computation of the abelianization of  $V(\Gamma_N, \Lambda_N,\ell)$ in \cite{stein1992groups} by examining the exact sequence in Theorem \ref{ah conj}. Let $d=gcd(n_1,...,n_k)$ It is clear the abelianisation surjects onto $ H_1(\G(n_1,..,n_k)= (\mathbb{Z}/d\mathbb{Z})^{k-1}$. Also, we have that the map $H_2(\G(n_1,..,n_k))=\mathbb{Z}_d^{^{k-1}C_2 } \rightarrow H_0(\G(n_1,..,n_k)) \otimes \mathbb{Z}_2=\mathbb{Z}_d \otimes \mathbb{Z}_2$ is the zero map whenever $d$ is odd, ensuring the surjection is an isomorphism.
  If $d$ is even, the image of the above map is $\mathbb{Z}_2=\mathbb{Z}_d \otimes \mathbb{Z}_2$. In this case, we get a split exact sequence that recovers Stein's computation of the abelianisation. 
\end{rmk}
\begin{rmk}[Left regular representations of Garside categories, Type $F_\infty$ simple groups]
In Stein's original paper \cite{stein1992groups}, she was able to show that the integral groups are of type $F_\infty$ and are virtually simple. It is interesting to remark that the realisation of Stein's groups as the topological full groups arising from certain $k$-graphs also gives rise to this fact, due to the Garside framework of Li \cite{li2021left}, one may alternatively apply [Theorem C, \cite{li2021left}] to reach the same conclusion. 
\end{rmk}
\begin{corollary}
    Let $N_1=\{n_1,...,n_k\}, N_2=\{m_1,...,m_j\}$ be two collections of integers. Let $\ell_1 \in \Gamma_{N_1}$, $\ell_2 \in \Gamma_{N_2}$. Then, 
    $$ \Gamma_{N_1}/N_{\Lambda_{N_1}} \cong 
 \Gamma_{N_2}/N_{\Lambda_{N_2}}\cong  \mathbb{Z}_{gcd(n_1-1,...,n_k-1)} \cong \mathbb{Z}_{gcd(m_1-1,...,m_j-1)},$$
 and $[\ell_1]=[\ell_2] \in \mathbb{Z}_{gcd(n_1-1,...,n_k-1)}$.
\end{corollary}

\end{document}